\documentclass[a4paper, 11pt, reqno]{amsart}
\usepackage{amsmath,amssymb,amsthm,amsfonts,mathrsfs,color,graphicx}
\usepackage{setspace}
\usepackage{ulem}
\usepackage[margin=3cm]{geometry}

\usepackage[hidelinks]{hyperref}
\usepackage{cite}
\usepackage{enumerate,latexsym}

\DeclareMathOperator{\rr}{\mathbb R}

\DeclareMathOperator{\dis}{dist}
\DeclareMathOperator{\spt}{spt}
\DeclareMathOperator{\id}{Id}
\DeclareMathOperator{\ze}{{\bf 0}}
\DeclareMathOperator{\ba}{B}
\DeclareMathOperator{\bo}{\partial B}
\DeclareMathOperator{\s}{S}
\DeclareMathOperator{\diver}{div}
\DeclareMathOperator{\hess}{Hess}

\newcommand{\R}{\mathbb{R}}

\newcommand{\si}{\Sigma}
\newcommand{\m}{{\bf M}}
\newcommand{\na}{\mathbb{N}}
\newcommand{\rt}{\mathbb{R}^{3}}
\newcommand{\de}{\mathcal{D}_}
\newcommand{\p}{\partial}
\newcommand{\ti}{\widetilde{T}}

\newcommand{\mres}{\mathbin{\vrule height 1.6ex depth 0pt width
0.13ex\vrule height 0.13ex depth 0pt width 1.3ex}}

\newtheorem{theorem}{Theorem}
\newtheorem*{thmA}{Theorem A}
\newtheorem*{thmB}{Theorem B}
\newtheorem*{thmC}{Theorem C}
\newtheorem{lemma}{Lemma}
\newtheorem{proposition}{Proposition}
\newtheorem{remark}{Remark}

\newtheorem{corollary}[theorem]{Corollary}
\newtheorem*{cor}{Corollary}
\theoremstyle{definition}\newtheorem{definition}{Definition}
\newtheorem{claim}{Claim}

\numberwithin{equation}{section}

\title{A two-piece property for free boundary minimal surfaces in the ball}
\author{Vanderson Lima and Ana Menezes}
\address{Instituto de Matem\'atica e Estat\'istica\\ Universidade Federal do Rio Grande do Sul\\ Brazil}
\email{vanderson.lima@ufrgs.br}

\address{Department of Mathematics\\ Princeton University\\ USA}
\email{amenezes@math.princeton.edu}

\begin{document}

\begin{abstract}
We prove that every plane passing through the origin divides an embedded compact free boundary minimal surface of the euclidean $3$-ball in exactly two connected surfaces. We also show that if a region in the ball has mean convex boundary and contains a nullhomologous diameter, then this region is a closed halfball. Moreover, we prove the regularity at the corners of currents minimizing a partially free boundary problem by following ideas by Gr\"uter and Simon. Our first result gives evidence to a conjecture by Fraser and Li. 
\end{abstract}

\maketitle

\section{Introduction}

A beautiful theorem by A. Ros \cite{R} states that every equator of the (round) $3$-sphere divides an embedded closed minimal surface in exactly two open connected pieces. An interesting fact is that this result can be seen as a consequence of a (still open) conjecture due to Yau - which says that the first nonzero eigenvalue of the Laplacian of an embedded closed minimal surface in the $3$-sphere is equal to 2 - together with the Courant nodal domain theorem. Hence, Ros's result can be seen as an evidence to the conjecture.

The analogy between the theory of closed minimal surfaces of the $3$-sphere and the theory of compact free boundary minimal surfaces of the unit euclidean 3-ball $\ba^3$ is well-known and has been well explored in many recent works, see for instance \cite{F.S1,F.S3,FraserLi,S,ACS,AN,L}. In this paper, inspired by this analogy, we prove the analog of Ros's result in the context of free boundary minimal surfaces.

\begin{thmA}[The two-piece property]
Every plane in $\rt$ passing through the origin divides an embedded compact free boundary minimal surface of the unit $3$-ball $\ba^3$ in exactly two connected surfaces.
\end{thmA}

To prove this theorem we need the following result which is also the analog of another result by Ros in \cite{R}.

\begin{thmB}
Let $W\subset \ba^3$ be a connected closed region with mean convex boundary such that $\partial W$ meets $\s^2$ orthogonally along its boundary and $\partial W$ is smooth. Suppose $W$ contains a straight line segment joining two antipodal points of $\s^2$, which is nullhomologous in $W$ (see Definition \ref{def-nul}). Then $W$ is a closed halfball.
\end{thmB}

We say that a surface $\si \subset \ba^3$ {\it links} a curve $\Gamma$, if $\si$ does not meet $\Gamma$ and it is homotopically non-trivial (relative to $\bo^3$) in $\ba^3\setminus \Gamma$ (see Figure \ref{fig-link}). An interesting consequence of Theorem B is the following corollary which is the analog of a result in $\s^3$ due to Solomon \cite{Sol}.

\begin{figure}[!h]
\includegraphics[scale=0.6]{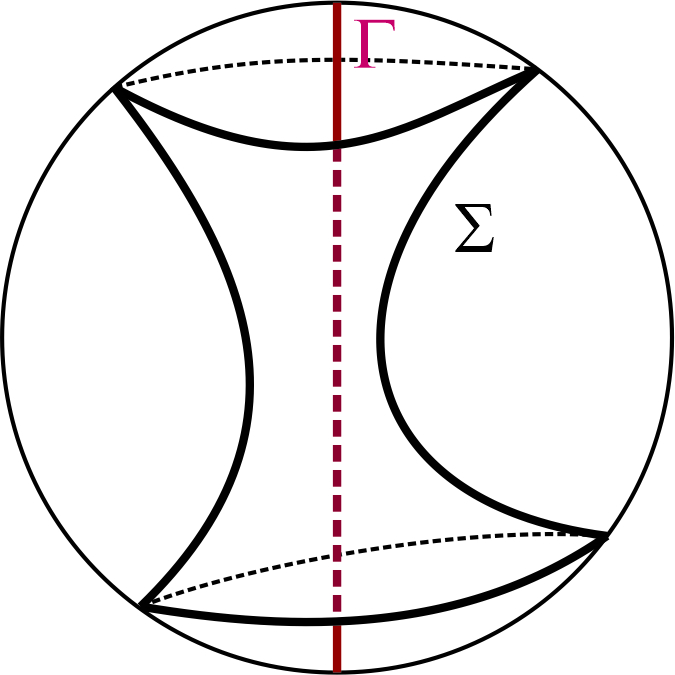}
\caption{The surface $\Sigma$ links the curve $\Gamma$.}
\label{fig-link}
\end{figure}

\begin{cor}
Every embedded compact free boundary minimal surface of $\ba^3$ either meets or links each straight line passing through the origin.
\end{cor}

In the proof of both Theorem A and Theorem B, we need the existence and regularity of a minimizer for a partially free boundary problem. Namely, let  $W \subset \rt$ be a compact domain such that $\partial W = S\cup M$, where $S$ is a compact $C^2$ surface (not necessarily connected) with boundary, $M$ is a smooth, compact mean convex surface with boundary, which intersects $S$ ortogonally along $\partial S$, and $\mathring{S}\cap\mathring M = \emptyset$ (here $\mathring A$ denotes the topological interior of $A$). Let $\gamma$ be a compact curve which is contained in $W$ and such that $\gamma\cap S$ is either empty or consists of a finite number of points (the corners); and consider the class of surfaces in  $W$ whose boundary minus $\gamma$ is contained in $S$. We look for a surface $\si$ which has least area among all such surfaces. The existence of such surface $\si$ follows from general compactness results about currents, and the regularity of $\si$ away from the corners can be proved using results of \cite{Gr2,HarSim,Ed} (see Section \ref{reg1}, Theorem \ref{thm.reg1}). It was reported in \cite{Gr3} that for a problem similar to this one, the regularity at the corners  would be settled in a work of Gr\"uter and Simon (unpublished). In Section \ref{sec-current} we give the details of this proof in the case where $\gamma$ intersects $S$ orthogonally by following the ideas contained in \cite{Gr3}. In particular, we prove the following regularity result.

\begin{thmC}
Suppose $\gamma$ intersects $S$ orthogonally and is $C^2$ except possibly at a finite number of points. Then the minimizer for the partially free boundary problem described above is a connected oriented embedded minimal surface which meets $S$ orthogonally and is $C^{1,\alpha}$, $0<\alpha<1,$ in a neighborhood of each corner and is $C^2$ away from the corners and the possible isolated singularities of $\gamma.$
\end{thmC}

The study of free boundary minimal surfaces (in euclidean domains) has attracted significant attention for several decades (see for instance classical works as \cite{Cou,HN} or more recent results as \cite{ACS,ACS2,AN,CFP,D,FGM,Li,LZ,L,MNS,Mc,R2,S,SZ,T,Y} and references therein). Recently there was an increase in interest for free boundary minimal surfaces in the unit euclidean 3-ball $\ba^3$ due to the work by Fraser and Schoen \cite{F.S1} (see also \cite{F.S2,F.S3}) where they made a connection between these objects and the Steklov eigenvalue problem. In analogy to Yau's conjecture mentioned above, Fraser and Li \cite{FraserLi} conjectured that the first nonzero Steklov eigenvalue of an embedded compact free boundary minimal surface in $\ba^3$ is equal to 1. This conjecture together with the Courant nodal domain theorem for the Steklov problem (stated for instance in \cite{GP}, Section 6) implies the two-piece property for free boundary minimal surfaces in $\ba^3$ (see Remark \ref{rem-nodal}). Hence, our result in Theorem A can be seen as an evidence to the conjecture by Fraser and Li.

In the last few years there have been many important studies about free boundary minimal surfaces. Ambrozio, Carlotto and Sharp \cite{ACS2} established compactness theorems for free boundary minimal hypersurfaces. Maximo, Nunes and Smith \cite{MNS} proved the existence of free boundary minimal annuli through a degree argument. Li and Zhou \cite{LZ} 
developed a min-max theory for free boundary minimal hypersurfaces.

We should mention that the class of free boundary minimal surfaces in $\ba^3$ is very rich. In fact, many techniques have been developed to construct new examples of free boundary minimal surfaces in $\ba^3$. For instance, Fraser and Schoen \cite{F.S1,F.S3} constructed examples with genus $0$ and any number of boundary components. Using gluing methods, Folha, Pacard and Zolotareva \cite{F.P.Z} constructed examples with genus $1$ and any large number of boundary components, and also obtained examples of genus $0$ and large number of boundary components displaying similar asymptotic behavior to Fraser-Schoen family. Examples with large genus and $3$ boundary components were constructed by Ketover \cite{K}, where he also obtained examples with the symmetry group of the Platonic solids, both using min-max methods. Kapouleas and Li \cite{K.L} also produced examples with large genus and $3$ boundary components, and examples with dihedral symmetry. Using gluing methods, Kapouleas and Wiygul \cite{K.W} constructed examples with one boundary component and large genus, converging to an equatorial disk with multiplicity 3, as the genus goes to infinity. More recently, Carlotto, Franz and Schulz \cite{CFS} applied min-max methods to prove the existence of embedded free boundary minimal surfaces with connected boundary and arbitrary genus.

This paper is organized as follows. In the second section we will prove both Theorem A and Theorem B; in Section \ref{sec-current} we will present the proof of the regularity at the corners of a minimizer for the partially free boundary problem mentioned above; and in Appendix A we will show an application of Serrin's Maximum Principle (Lemma 2, \cite{Ser}) at a corner. \\

\noindent
{\bf Acknowledgements}: The first author would like to thank Princeton University for the hospitality where part of the research and preparation of this article were conducted. The authors would like to thank the anonymous referee for valuable suggestions.

\section{The two-piece property and other results}
\label{sec.main}
Throughout the paper we say that a curve in $\ba^3$ is a {\it diameter} if it is a straight line segment joining two antipodal points of $\s^2=\bo^3;$ and we will define an {\it equatorial disk} as the intersection of $\ba^3$ with a plane passing through the origin.

Given a surface $\Sigma$ in $\ba^3$, we will write its boundary as $\p\Sigma=\gamma_I\cup\gamma_S$ where $int (\gamma_I)\subset int (\ba^3)$ and $\gamma_S\subset \s^2.$

\begin{definition}
Let $\Sigma$ be a compact surface properly immersed in $\ba^3.$ We say that $\Sigma$ is a {\it minimal surface with free boundary} if the mean curvature vector of $\Sigma$ vanishes and $\Sigma$ meets $\s^2$ orthogonally along $\p \Sigma$ (in particular, $\gamma_I=\emptyset).$ We say that $\Sigma$ is a {\it minimal surface with partially free boundary} if the mean curvature vector of $\Sigma$ vanishes and its boundary $\p \Sigma=\gamma_I\cup\gamma_S$ satisfies that $\Sigma$ meets $\s^2$ orthogonally along $\gamma_S$ and $\gamma_I\neq\emptyset$.\end{definition}

From now on, given a (partially) free boundary minimal surface $\Sigma\subset \ba^3$ with boundary $\p \Sigma=\gamma_I\cup \gamma_S$, we will call $\gamma_I$ its fixed boundary and $\gamma_S$ its free boundary.

\begin{lemma}\label{lem.n-prong}
Let $\Sigma_1$ and $\Sigma_2$ be two distinct (partially) free boundary minimal surfaces with boundary $\partial\Sigma_i= \gamma_{I}^i\cup\gamma_{\s}^i$ that are tangent at a point $p\in\Sigma_1\cap\Sigma_2.$ Then

\begin{enumerate}
\item if $p\in \mbox{int}(\Sigma_1)\cap \mbox{int}(\Sigma_2)$, there exists a neighborhood of $p$ where the intersection $\Sigma_1\cap\Sigma_2$ is given by $2l$ curves, $l\geq 2$, starting at $p$ and making equal angle. See Figure \ref{n-prong}(a);

\item if $p\in int(\gamma_{\s}^1)\cap int(\gamma_{\s}^2),$ there exists a neighborhood of $p$ where the intersection $\Sigma_1\cap \Sigma_2$ is given by $k$ curves, $k\geq1$, starting at $p$. See Figures \ref{n-prong}(b), \ref{n-prong}(c), \ref{n-prong}(d).
\end{enumerate}
In both cases, $p$ is called an $n$-prong singularity.
\end{lemma}

\begin{proof}
See ~\cite[Lemma 1.4]{FHS} for the proof of $(1)$. For $(2)$, it is known that any free boundary minimal surface can be extended analytically as a minimal surface in a neighborhood of each point of its free boundary (see Theorems 2 and 2' in \cite{DHT}, pg. 178). So we can extend $\si_1$ and $\si_2$ on a neighborhood of $p$. Denote by $\tilde{\si}_i$, $i=1,2$, the extended surface from $\Sigma_i$. In particular, $\tilde{\si}_1$ and $\tilde{\si}_2$ are two minimal surfaces tangent at an interior point. Then, by item $(1)$, $\tilde{\si}_1\cap\tilde{\si}_2 $ is given locally by $2\ell$ curves, $\ell\geq 2$,  {\it starting} at $p$ and making equal angle. Denote these curves by $\alpha_1, \alpha_2, \cdots, \alpha_{2\ell}$ and by $v_1, v_2, \cdots, v_{2\ell}\in T_p{\tilde\Sigma_i}$ their tangent vectors at time zero. Denote by $\theta$ the (constant) angle between these vectors, and observe that $\theta\leq\pi/2$ and $\ell\theta=\pi$. Let $V$ be the half tangent plane of $\Sigma_i$ at the boundary point $p$, that is, $V$ is the set of vectors $v$ such that there exists a curve $\alpha\subset \Sigma_i$ with $\alpha'(0)=v.$ To simplify our notation, let us assume, without loss of generality, that $V$ is the half plane $\{y\geq0\}$, $p=(0,0,0)$, the sphere is centered at (0,1,0), and (after possibly reordering the vectors) $v_1$ is the vector with least angle from the positive $x$-axis (counterclockwise). In particular, $\angle (v_i, \  \mbox{positive} \  x\mbox{-axis})=\angle (v_1, \  \mbox{positive} \  x\mbox{-axis})+\theta(i-1)$. 
 
First observe that if $0<\angle (v_i, \  \mbox{positive} \  x\mbox{-axis})<\pi,$ then the vector $v_i$ is strictly pointing inside the sphere and, consequently,  the arc $\alpha_i$ is (locally) contained  in $\Sigma_i.$ 

Suppose $\angle (v_1, \  \mbox{positive} \  x\mbox{-axis})=0$. Notice that since $\gamma_{\s}^i$ are real analytic, locally the intersection $\gamma_{\s}^1\cap \gamma_{\s}^2$ is either only the point $p$ or they coincide. If $\gamma_{\s}^1\cap \gamma_{\s}^2=\{p\}$ locally, then, by analyticity, locally the arc $\alpha_1$ is necessarily either contained outside the sphere or inside the sphere; that is, the arc $\alpha_1\setminus\{p\}$ is either contained in $\tilde{\si}_i\setminus \Sigma_i$ or in $\si_i$. If $\gamma_{\s}^1=\gamma_{\s}^2$ locally, then the arc $\alpha_1$ has to coincide with them. Let us remark that the same conclusions hold if $\angle (v_i, \  \mbox{positive} \  x\mbox{-axis})=\pi$, for some $i.$

Since we have at least four arcs $\alpha_i$ and $\ell\theta = \pi$, there is at least one vector $v_i$ with $0<\angle (v_i, \  \mbox{positive} \  x\mbox{-axis})<\pi.$ Therefore, the intersection $\Sigma_1\cap \Sigma_2$ is given by $k$ curves, $k\geq1$, starting at $p$.
\end{proof}
\begin{figure}[!h]
\includegraphics[scale=0.65]{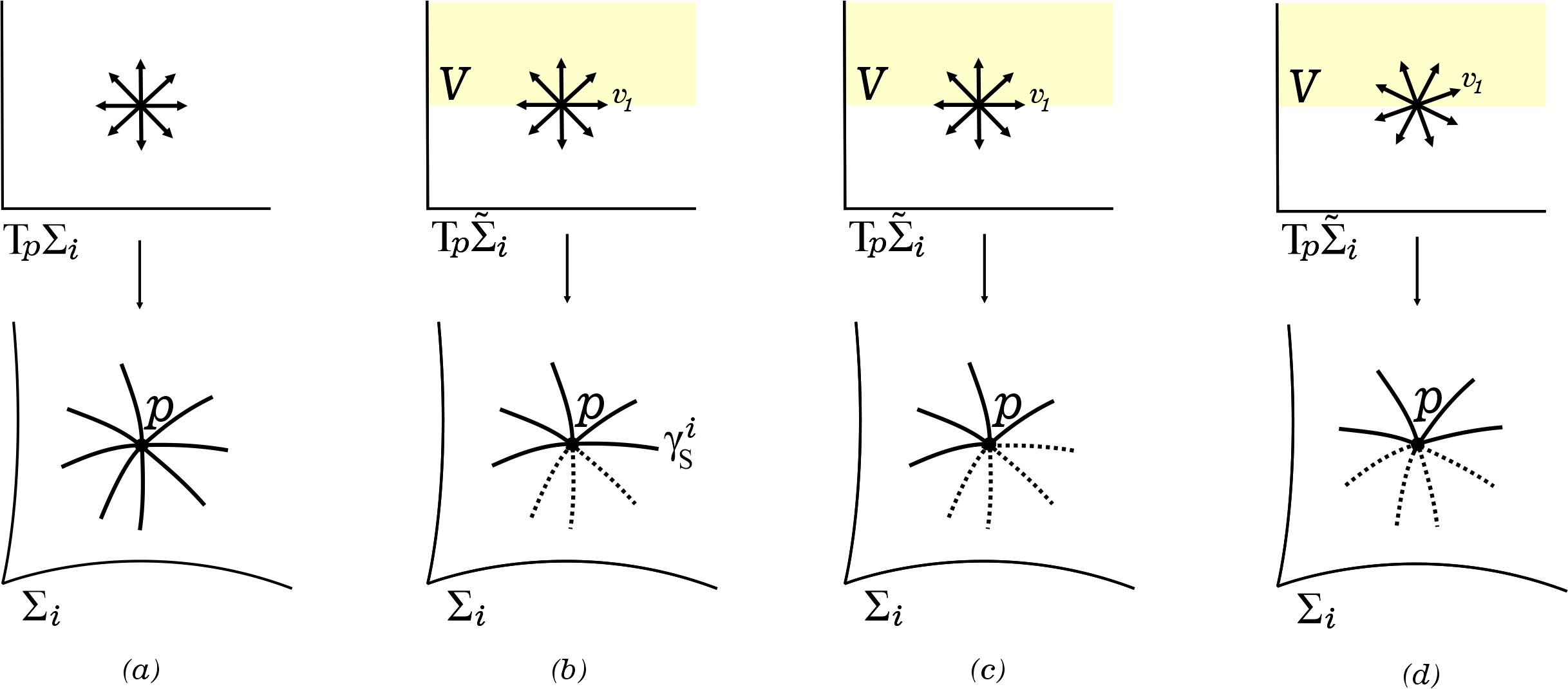}
\caption{Examples of n-prong singularities. Figure $(a)$: $p\in \mbox{int}\si_i$. Figure $(b)$: $\gamma_S^1$ and $\gamma_S^2$ locally coincide. Figure $(c)$: $\gamma_S^1\cap\gamma_S^2=\{p\}$ locally, $\angle (v_i, \ x\mbox{-axis})=0,$ $i=1,4,$ with $\alpha_1\subset \tilde{\si_i}\setminus\si_i$ and $\alpha_4\subset\si_i$. Figure $(d)$: $\angle (v_1, \ x\mbox{-axis})>0$.}
\label{n-prong}
\end{figure}

\begin{definition}\label{def-stable}
Let $\Sigma$ be a partially free boundary minimal surface in $\ba^3$ with piecewise smooth boundary $\p \Sigma=\gamma_I\cup \gamma_S$. We say that $\Sigma$ is {\it stable} if for any function $f\in C^{\infty}(\Sigma)$ such that $f|_{\gamma_I}\equiv 0$ we have
\begin{equation}\label{eq-stable}
-\int_{\Sigma} (f\Delta_{\Sigma}f+|A_\Sigma|^2f^2)d\Sigma+\int_{\gamma_S}\left(f\frac{\p f}{\p \nu}-f^2\right)ds\geq 0,
\end{equation}
or equivalently
\begin{equation}\label{eq-stable2}
\int_{\Sigma} (|\nabla_\Sigma f|^2-|A_\Sigma|^2f^2)d\Sigma-\int_{\gamma_S}f^2ds\geq 0,
\end{equation}
where $\nu$ is the outward normal vector field to $\gamma_S$.
\end{definition}


\begin{lemma}\label{lem2}
Let $\Sigma$ be a compact orientable immersed partially free boundary stable minimal surface in $\ba^3$ with piecewise smooth boundary $\p \Sigma=\gamma_I\cup \gamma_S$.
Suppose $\gamma_I$ is contained in an equatorial disk. Then $\Sigma$ is totally geodesic.  The same result holds in the case where $\Sigma$ has finite area and isolated singularities on $\gamma_I$.
\end{lemma}

\begin{proof}
Let $\Sigma$ be as in the hypotheses and denote by $D$ the equatorial disk that contains its fixed boundary $\gamma_I$. 

Let us  first assume that the fixed boundary of $\Sigma$ does not have singularities.

Let $v\in \s^2$ be a vector orthogonal to the disk $D$ and consider the function $f(x)=\langle x,v\rangle$, $x \in \Sigma.$ By hypothesis we know that $f|_{\gamma_I}\equiv 0$, so (\ref{eq-stable}) holds. Moreover, since $\Sigma$ is minimal, it is well-known that 
\begin{equation}\label{eq1}
\Delta_\Sigma f(x)=0.
\end{equation}


On the other hand, 
\begin{equation}\label{eq2}
\frac{\p f}{\p \nu}(x)=\nu\langle x,v\rangle = \langle\nabla^{\R^3}_\nu x, v\rangle=\langle\nu(x), v\rangle=\langle x, v\rangle =f(x),
\end{equation}
since $\Sigma$ is free boundary on $\gamma_S$.

Using \eqref{eq1} and \eqref{eq2} in \eqref{eq-stable}, we get
\begin{equation}\label{eq3}
|A_\Sigma|^2(x)\langle x,v\rangle^2= 0 \ \ \mbox{for any} \ \ x\in\Sigma.
\end{equation}

If $|A_\Sigma|\equiv0$ then $\Sigma$ is totally geodesic and we are done.

If $|A_\Sigma|(x)>0$ for some $x\in\Sigma,$ then we can find a neighborhood $U$ of $x$ in $\Sigma$ such that $|A_{\Sigma}|$ is strictly positive. By \eqref{eq3}, this implies $\langle y, v\rangle=0$ for any $y\in U$, that is, $U$ is contained in the disk $D$. Therefore, $\Sigma$ is entirely contained in the disk $D;$ in particular, it is totally geodesic.

Now let us suppose that  $\Sigma$ has isolated singularities in the fixed boundary (which is contained in the equatorial disk). Let us consider a cut-off function $\eta_\epsilon:[-1,1]\to [0,1]$ so that
\begin{itemize}
\item $\eta_\epsilon(s)=0$ for $|s|<\epsilon$,
\item $\eta_\epsilon(s)=1$ for $|s|>2\epsilon$,
\item $|\eta^{'}_\epsilon|<\displaystyle\frac{C}{\epsilon}$, for some constant $C;$
\end{itemize}
and define $\phi_\epsilon:\Sigma\to [0,1]$ as $\phi_\epsilon(x)=\eta_\epsilon(f(x)),$ where $f(x)=\langle x, v\rangle.$ In particular, we have $|\nabla_\Sigma\phi_\epsilon|<C/\epsilon.$ 

Now let us take the function $f_\epsilon=\phi_\epsilon f$. It satisfies $f_\epsilon|_{\gamma_I}\equiv 0$ and so (\ref{eq-stable2}) holds.

Observe that 
$$|\nabla_\Sigma f_\epsilon|^2=\phi_\epsilon^2|\nabla_\Sigma f|^2+2f\phi_\epsilon\langle \nabla_\Sigma f, \nabla_\Sigma \phi_\epsilon\rangle + f^2|\nabla_\Sigma \phi_\epsilon|^2$$
and
$$\begin{array}{rcl}
\displaystyle\int_\Sigma \phi^2_\epsilon|\nabla_\Sigma f|^2d\Sigma &=& \displaystyle-\int_\Sigma f\phi_{\epsilon}^2\Delta_\Sigma f d\Sigma-\int_\Sigma 2f\phi_\epsilon\langle\nabla_\Sigma \phi_\epsilon, \nabla_\Sigma f\rangle d\Sigma+\int_{\p \Sigma}\phi_\epsilon^2 f\frac{\p f}{\p\nu}ds\\
\\
\displaystyle&=&\displaystyle -\int_\Sigma 2f\phi_\epsilon\langle\nabla_\Sigma \phi_\epsilon, \nabla_\Sigma f\rangle d\Sigma+\int_{\gamma_S}\phi_\epsilon^2 f^2ds,
\end{array}$$
since $\Delta_\Sigma f\equiv 0$, $\frac{\p f}{\p \nu}=f$ and $f|_{\gamma_I}\equiv 0.$

Hence, applying it to (\ref{eq-stable2}), we get
\begin{equation}\label{eq4}
\int_\Sigma  (f^2|\nabla_\Sigma \phi_\epsilon|^2 -|A_\Sigma|^2\phi_\epsilon^2 f^2)d\Sigma \geq0.
\end{equation}

Since $\Sigma$ has finite area and $|\nabla_\Sigma \phi_\epsilon|<C/\epsilon$, we have

\begin{equation}\label{eq5}
\int_\Sigma  f^2|\nabla_\Sigma \phi_\epsilon|^2d\Sigma < 4\epsilon^2\frac{C^2}{\epsilon^2}\textrm{Area}\bigl(\Sigma\cap\{|f|^{-1}(\epsilon, 2\epsilon)\}\bigr)\to 0 \ \mbox{as} \ \epsilon\to0.
\end{equation}

 Then, since $\phi_\epsilon\to 1$ as $\epsilon\to 0,$ \eqref{eq5} together with \eqref{eq4} yield
$$
\int_\Sigma|A_\Sigma|^2 f^2d\Sigma=0.
$$

Therefore, we get the same conclusions as above.
\end{proof}

\begin{remark}\label{rem-lem2}
Observe that by its proof, in order to be able to apply Lemma \ref{lem2}, we just need regularity and stability of the surface outside the equatorial disk where the fixed boundary is contained.
\end{remark}

\begin{remark}\label{rem-nodal}
Let $M$ be an embedded free boundary minimal surface in $\ba^3.$ Recall that a nodal domain of a function is a maximally connected subset of the domain where the function does not change sign, and the Courant nodal domain theorem for the Steklov problem says that an eigenfunction corresponding to the $n$-th nonzero Steklov eigenvalue has at most $n+1$ nodal domains. Let $P$ be a plane passing through the origin and let $v\in\s^2$ be a vector orthogonal to $P$. The Jacobi function $f: M\to \rr, f(x)=\langle x,v\rangle$,  defined in the proof of Lemma \ref{lem2}, is an eigenfunction with eigenvalue 1 for the Steklov problem. Hence, assuming Fraser-Li conjecture, it follows that $f$ has at most two nodal domains. Moreover, we can use the (interior and boundary) maximum principle with equatorial disks to conclude that $f$ has in fact two nodal domains, that is, the plane $P$ divides $M$ in excatly two connected surfaces.  Hence, our result in Theorem \ref{thm-main} can be seen as an evidence to the conjecture by Fraser and Li.
\end{remark}

\begin{definition}
Let $W$ be a region in $\ba^3$ and let $\alpha\subset W$ be a diameter. We say that $\alpha$ is {\it nullhomologous} in $W$ if there exists a compact surface $M \subset W$ such that $\partial M = \alpha\cup\gamma$, where $\gamma \subset \s^2$ (see Figure \ref{fig-null}).
\label{def-nul}
\end{definition}

\begin{figure}[!h]
\centering
\includegraphics[scale=0.6]{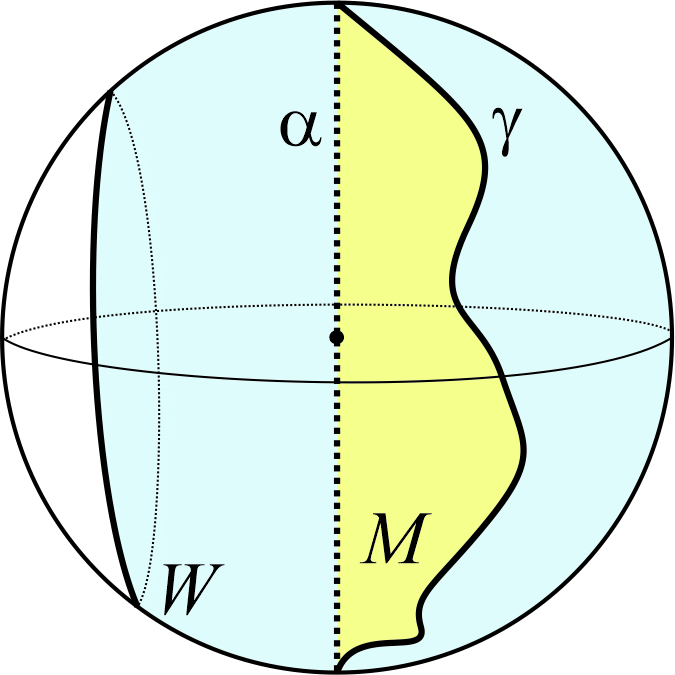}
\caption{In this region $W$, any diameter $\alpha\subset W$ is nullhomologous.}
\label{fig-null}
\end{figure}

The boundary of the region $W$ can be written as $U\cup V$, where $int (U)\subset int (\ba^3)$ and $V\subset \s^2.$ In the next theorem we will denote by $\p W$ the closure of the component $U,$ that is, $\p W=\overline U.$ 


\begin{theorem}
Let $W\subset \ba^3$ be a connected closed region with (non-strictly) mean convex boundary such that $\partial W$ meets $\s^2$ orthogonally along its boundary and $\partial W$ is smooth. If $W$ contains a diameter $\alpha$, and $\alpha$ is nullhomologous in $W$, then $W$ is a closed halfball.
\label{thm-nul}
\end{theorem}

\begin{proof}
Up to a rotation of $\alpha$ around the origin, we can assume that $\alpha\cap \p W$ is nonempty. Since $\alpha$ is nullhomologous in $W$, there is at least one curve $\gamma\subset \s^2$ such that there exists a surface contained in $W$ with boundary $\Gamma=\alpha\cup \gamma$. We consider the class of admissible currents
\begin{eqnarray*}
\mathfrak{C} = \{T \in \de 2(\rt); \ T \ \mbox{is integer multiplicity rectifiable},\\
\spt T \subset W \ \mbox{and is compact}, \ \mbox{and} \ \spt \bigl([[\alpha]] - \p T\bigr) \subset \s^2\cap W\},
\end{eqnarray*}
where $[[\alpha]]$ is the current associated to $\alpha$, and we minimize area (mass) in $\mathfrak{C}.$ Then, by the results presented in Section \ref{pfb}, we get a compact embedded (orientable) partially free boundary minimal surface $\Sigma\subset W$ which minimizes area among compact surfaces in $W$ with boundary on the class $\Gamma=\alpha\cup\gamma;$ in particular, its fixed boundary is exactly $\alpha.$ Moreover, by Proposition \ref{max.princ} in Section \ref{pfb}, either $\Sigma\subset\p W$ or $\Sigma\cap \p W\subset \alpha.$



\begin{claim}
\label{claim1}
$\Sigma$ is stable.
\end{claim}
In the case $\partial W\cap\si \subset \Gamma$, the surface $\si$ is automatically stable in the sense of Definition \ref{def-stable}, since it minimizes area for all local deformations. 
Suppose $\si \subset \p W$. For any $f\in C^\infty(\Sigma)$ with $f|_{\alpha}\equiv0$, consider $Q(f,f)$ defined by
$$Q(f,f)=\frac{\int_\Sigma\left(|\nabla_\Sigma f|^2-|A_\Sigma|^2f^2\right)d\Sigma-\int_{\gamma}f^2ds}{\int_{\si}f^2 d\si},$$ 
and let $f_1$ be a first eigenfunction, i.e., $Q(f_1,f_1) = \inf_{f} Q(f,f)$.

Observe that although differently from the classical stability quotient (we have an extra term that depends on the boundary of $\Sigma$) we can still guarantee the existence of a first eigenfunction. In fact, since for any $\delta>0$ there exists $C_\delta>0$ such that $||f||_{L^2(\partial\Sigma)} \leq \delta ||\nabla f||_{L^2(\Sigma)}+ C_\delta ||f||_{L^2(\Sigma)}$, for any $f\in W^{1,2}(\Sigma)$, we can use this inequality to prove that the infimum is finite. Once this is established the classical arguments to show the existence of a first eigenfunction work.

Since $|\nabla |f_1||=|\nabla f_1|$ a.e., we have $Q(f_1,f_1)=Q(|f_1|,|f_1|)$, that is, $|f_1|$ is also a first eigenfunction. Since $|f_1|\geq 0$, the maximum principle implies that $|f_1|>0$ in $\Sigma\setminus\partial\si$, in particular, $f_1$ does not change sign in $\Sigma\setminus\partial\si$. Then we can assume that $f_1>0$ in $\Sigma\setminus\partial\si$ and, by continuity, we get $f_1\geq0$ in $\gamma.$ Therefore, we can use $f_1$ as a test function to our variational problem:
 Let $\zeta$ be a smooth vector field such that $\zeta(x) \in T_x\s^2,$ for all $ x \in \s^2$, $\zeta(x)\in (T_x\si)^\perp,$ for all $x \in \si$, and $\zeta$ points towards $W$ along $\si$. Let $\Phi$ be the flow of $\zeta$. For $\varepsilon$ small enough the surfaces $\si_t = \{\Phi\bigl(x,tf_1\bigr)$; $x \in \si$, $0 < t < \varepsilon\}$ are contained in $W$. Since $\Sigma$ has least area among the surfaces $\si_t$, we know that 
$$0 \leq \frac{d^2}{dt^2}\biggl|_{t=0^+}|\si_t| = \int_{\Sigma} (|\nabla_\Sigma f_1|^2-|A_\Sigma|^2f_1^2)d\Sigma-\int_{\gamma}f_1^2ds,$$
which implies that $Q(f_1,f_1)\geq0$. Since $f_1$ is a first eigenfunction, we get that $Q(f,f)\geq0$ for any $f\in C^\infty(\Sigma)$ with $f|_{\alpha}\equiv0$. Therefore, we have stability for $\Sigma$.

Then, since $\alpha$ is contained in an equatorial disk, Lemma \ref{lem2} implies that $\Sigma$ is necessarily a half disk. If  $\Sigma\subset\p W$, then we already conclude that $W$ has to be a halfball.

Suppose $\Sigma\cap \p W\subset \alpha$. Rotate $\Sigma$ around $\alpha$ until the last time it remains in $W$ (this last time exists once $\Sigma\cap \p W$ is non empty), and let us still denote this rotated surface by $\Sigma$. In particular, there exists a point $p$ where $\Sigma$ and $\p W$ are tangent. We will conclude that $W$ is necessarily a halfball.

In fact, if $p\in \mbox{int}(\alpha),$ we can write $\p W$ locally as a graph over $\si$ around $p$ and apply the classical Hopf Lemma; if $p\in \p \alpha,$ we can use the Serrin's Maximum Principle at a corner (see Appendix \ref{app} for the details); and if $p\in \Sigma\setminus\alpha$ we can apply (the interior or the free boundary version of) the maximum principle. In any case, we get that $W$ is a halfball.

\end{proof}

An equatorial disk $D$ divides the ball $\ba^3$ into two (open) halfballs. We will denote these two halfballs by $\ba^+$ and $\ba^-,$ and we have $\ba^3\setminus D = \ba^+\cup \ba^-.$

In the next proposition we will summarize some simple facts about partially free boundary minimal surfaces in $\ba^3$ which we will use in the proof of Theorem \ref{thm-main}.

\begin{proposition}
\begin{enumerate}
\item[(i)] Let $D$ be an equatorial disk and let $\Sigma$ be a partially free boundary minimal surface in $\ba^3$ contained in one of the closed halfballs determined by $D$ and such that $\gamma_I\subset D$ (if $\gamma_I \neq \emptyset$). If $\Sigma$ is not an equatorial disk, then $\Sigma$ has necessarily nonempty fixed boundary and nonempty free boundary.

\item[(ii)] The only (partially) free boundary minimal surface that contains an arc segment of a great circle in its free boundary is (contained in) an equatorial disk.
\end{enumerate}
\label{prop-simple}
\end{proposition}

\begin{proof}
${\it (i)}$  If the free boundary were empty, we could apply the (interior) maximum principle with the family of planes parallel to the disk $D$ and conclude that $\Sigma$ should be a disk. On the other hand, if the fixed boundary were empty, then we would have a minimal surface entirely contained in a halfball without fixed boundary; hence, we could apply the (interior or free boundary version of) maximum principle with the family of equatorial disks that are rotations of $D$ around a diameter and conclude that $\Sigma$ should be a disk as well.

${\it (ii)}$ Let $D$ be an equatorial disk and suppose that $\Sigma$ is a (partially) free boundary minimal surface such that $\Sigma\cap D$ contains an arc segment $\alpha$ in $\s^2;$ in particular, since they are both free boundary, we know they are tangent along $\alpha.$ Hence, given a point $x\in \alpha$ there exists a neighborhood $U$ of $x$ in $D$ where either $\Sigma$ is on one side of $D$ or $\Sigma\cap U\setminus\alpha$ is given by a collection of $k$ curves,  $k\geq 1$, starting at $x$ (see Lemma \ref{lem.n-prong}). In this last case, for any point in $(\alpha\cap U)\setminus \{x\}$, we will have a neighborhood where $\Sigma$ is on one side of $D$; therefore, in either case, applying the boundary maximum principle we can conclude that $\Sigma$ should be (contained in) an equatorial disk.
\end{proof}

\begin{remark}
\label{ortho}
If $\Sigma_1$ and $\Sigma_2$ are two partially free boundary minimal surfaces in $\ba^3$ that intersect at a point $p\in\p \Sigma_1\cap \p\Sigma_2\cap \s^2$ transversally, then the intersection is locally given by a simple curve that meets $\s^2$ orthogonally at $p$.  In fact, in the same way as we argued in the proof of Lemma \ref{lem.n-prong}, item $(2)$, we can show that the intersection is locally given by a simple curve $\gamma$ that meets $\s^2$ at $p=\gamma(0)$. Let $\nu_i$ be the normal vector to $\Sigma_i$, $i=1,2.$  Since they meet transversally, we know $span(\nu_1,\nu_2)=T_p\s^2$ and,  since $\gamma\subset \Sigma_1\cap\Sigma_2,$ $\gamma'(0)$ is orthogonal to both $\nu_1, \nu_2$. Therefore, $\gamma$ meets $\s^2$ orthogonally at $p$. 
\end{remark}

Now we can prove the two-piece property for free boundary minimal surfaces in $\ba^3$.

\begin{theorem}
Let $M$ be a compact embedded free boundary minimal surface in $\ba^3.$ Then for any equatorial disk $D$, $M\cap \ba^+$ and $M\cap \ba^-$ are connected.
\label{thm-main}
\end{theorem}

\begin{proof}
If $M$ is an equatorial disk, then the result is trivial. So let us assume this is not the case.

Suppose that, for some equatorial disk $D$, $M\cap \ba^+$ is a disjoint union of two nonempty open surfaces $M_1$ and $M_2$, $M_1$ being connected. Notice that by Proposition \ref{prop-simple}(i) both $M_1$ and (all components of) $M_2$ have non empty fixed boundary and non empty free boundary.

Let us denote by $\Gamma$ the boundary of $M_1,$ which is not necessarily connected. We can write $\Gamma = \gamma_I\cup\gamma_S$, where $\gamma_I$ is its fixed boundary ($int(\gamma_I)\subset int(D)$) and $\gamma_S$ is its free boundary ($\gamma_S\subset \s^2)$. Since $M$ and $D$ are two distinct minimal surfaces, either $M$ and $D$ are transverse or the intersection $M\cap D$ contains at least one $n$-prong singularity (see Lemma \ref{lem.n-prong}).

Observe that, by applying (either the interior or free boundary version) of the maximum principle, we know that $M\cap D$ does not contain any isolated point in $D$ and, by Proposition \ref{prop-simple}(ii), $M\cap D$ does not contain any arc segment in $\s^2$.

Denote by $W$ and $W'$ the closures of the two components of $\ba^3\setminus M.$ They are compact domains with mean convex boundary, and observe that the curve $\Gamma$ is the boundary of an orientable surface contained in them (in fact, $M_1$ is orientable and $M_1\subset W, W'$).  Hence, we can minimize area for the following partially free boundary problem (see Section \ref{reg1}):

We consider the class of admissible currents
\begin{eqnarray*}
\mathfrak{C} = \{T \in \de 2(\rt); \ T \ \mbox{is integer multiplicity rectifiable},\\
\spt T \subset W \ \mbox{and is compact}, \ \mbox{and} \ \spt \bigl([[\gamma_I]] - \p T\bigr) \subset \s^2\cap W\},
\end{eqnarray*}
where $[[\gamma_I]]$ is the current associated to $\gamma_I$, and we minimize area (mass) in $\mathfrak{C}.$ Then, by the results presented in Section \ref{pfb}, we get a compact embedded (orientable) partially free boundary minimal surface $\Sigma\subset W$ with fixed boundary $\gamma_I$ and with possible isolated singularities in $\gamma_I\subset D$ (see Theorem \ref{last} and Remark \ref{last}). Moreover, by Proposition \ref{max.princ} in Section \ref{pfb}, either $\Sigma\subset\p W$ or $\Sigma\cap \p W\subset \alpha.$


Arguing as in Claim \ref{claim1} of Theorem \ref{thm-nul}, we can prove the stability of $\Sigma$ away from the disk $D$.  
Since in the proof of Lemma \ref{lem2}, for the case where $\Sigma$ has isolated singularities, we only use stability away from the disk, we can still get the conclusion from Lemma \ref{lem2}, that is, each component of $\Sigma$ is a piece of an equatorial disk.
The case $\si \subset \p W$ can not happen because this would imply that $M$ is a disk, and we are assuming it is not. Therefore, only the second case can happen, that is, any component of $\Sigma$ meets $\p W$ only at points of $\Gamma.$ Observe that each component of $\Sigma$ that is not bounded by a diameter is necessarily contained in $D$. If some component of $\Sigma$ were bounded by a diameter, then we could apply Theorem \ref{thm-nul} and would conclude that $M$ is an equatorial disk, which is not the case. Then $\Sigma$ is entirely contained in $D$ and, since $\Sigma\cap \p W\subset \Gamma$, $M\subset \p W$ and $M\cap D$ does not contain any segment on $\s^2$, we have $\Sigma\cap M=\gamma_I.$

Doing the same procedure as in the last paragraph for $W'$, we can construct another compact surface $\Sigma'$ of $D$ with fixed boundary $\p \Sigma'=\gamma_I$ and such that $\Sigma'\subset W'$ and $\Sigma'\cap M=\gamma_I$. Notice that $\Sigma\cup \Sigma'$ is a surface without fixed boundary of $D$, therefore $\Sigma\cup \Sigma'=D.$ In particular, $M\cap D=\gamma_I,$ which implies that $M_2=M\cap \ba^+\setminus M_1$ has no fixed boundary, a contradiction (by Proposition \ref{prop-simple}(i)). Therefore, the theorem is proved.
\end{proof}

\begin{corollary}
Every embedded compact free boundary minimal surface $M$ of $\ba^3$ either meets or links each diameter.
\end{corollary}

\begin{proof}
Let $\alpha$ be a diameter with endpoints $p,q$, and suppose $M$ does not meet $\alpha$. Write $\ba^3\setminus M = W\cup W'$, where $W$ contains $\alpha$, and the decomposition is disjoint. Suppose, by contradiction, that $M$ is homotopically trivial (relative to $\bo^3$) in $\ba^3\setminus\alpha$. We will prove that $\alpha$ is nullhomologous in $W$. 

In fact, since $M$ is homotopically trivial in $\ba^3\setminus\alpha$, we have that $\partial M\cap\s^2$ is a finite collection of simple closed curves in $\s^2$ which are homotopically trivial in $\s^2\setminus\{p,q\}$. Hence, there is a curve $\gamma \subset W\cap\s^2$ joining $p$ and $q$; and $\alpha\cup\gamma$ bounds a (topological) disk $V$ in $B^3$. If $V \subset W$, we are done. If that is not the case, by deforming $V$, if necessary, we can suppose that $V$ and $M$ are transverse. Let $N$ be the unit normal vector field to $M$ pointing into $W$, and denote by $M_{\epsilon}$ the intersection of $\ba^3$ with the boundary of a one-sided tubular neighborhood of $M$ (in the direction of $N$) of radius $\epsilon$. We can choose $\epsilon$ small enough such that $\partial M_{\epsilon}\cap(\alpha\cup\gamma) = \emptyset$, and $M_{\epsilon}$ is transverse to $V$. The intersection $M_{\epsilon}\cap V$ consists of a finite number of simple closed curves which bound open discs $U_1,\cdots,U_n$ in $V$ (since $V$ is a disk). The set $\si = M_{\epsilon}\cup\bigl(V\setminus\cup_{i=1}^{n}U_i\bigr)$ is a topological surface with $\alpha \subset \si \subset W$ and $\partial\si\setminus\alpha \subset \s^2$. Therefore $\alpha$ is nullhomologous in $W$.

Then, by Theorem \ref{thm-nul}, we conclude that $W$ is a closed halfball; in particular, $M$ is an equatorial disk. However, this contradicts the fact that $M\cap \alpha=\emptyset$. Therefore, $M$ links $\alpha$ necessarily.
\end{proof}

\section{Solution to a partially free boundary problem}\label{pfb}
\label{sec-current}

\subsection{Terminology}

Let $U \subset \R^{n + k}$ be an open set. We define
$$\mathcal{D}^n(U) = \{C^{\infty}\textrm{-} \ n\textrm{-}\textrm{forms} \ \omega ;\ \spt \ \omega \subset U\}$$
with the usual topology of uniform convergence of all derivatives on compact subsets. Its dual space is denoted by $\de n(U)$ and the elements of $\de n(U)$ are called $n$-currents in $U$. If $T \in \de n(U)$, and $W \subset U$ is open, the mass of $T$ in $W$ is defined by
$$\m_W(T) := \sup \{T(\omega); \ \omega \in \mathcal{D}^n(U), \ \spt \omega \subset W, \ |\omega| \leq 1\} \leq +\infty.$$

The boundary of $T$ is the $(n - 1)$-current $\p T \in \de {n-1}(U)$ given by 
$$\p T(\omega) := T(d\omega),$$
where $d$ denotes the exterior derivative operator.

Given a sequence $\{T_j\}_{j \in \na}$ in $\de n(U)$, we say that $T_j$ converges to $T \in \de n(U)$ as $j \to \infty$, if
$$T_j(\omega) \to T(\omega), \ \mbox{as} \ j \to \infty, \ \forall \ \omega \in \mathcal{D}^n(U).$$ 

Let $\mathcal{H}^n$ denote the $n$-dimensional Hausdorff measure. A set $M \subset \R^{n+k}$ is called countably $n$-rectifiable if $M$ is $\mathcal{H}^n$-measurable and if
$$M \subset \bigcup_{j = 0}^{\infty} M_{j},$$
where $\mathcal{H}^n(M_0) = 0$ and for $j \geq 1$, $M_j$ is an $n$-dimensional $C^1$-submanifold of $\R^{n+k}$. Such $M$ possesses $\mathcal{H}^n$-a.e. an approximate tangent space $T_{x}M$.

A current $T \in \de n(U)$ is called integer multiplicity rectifiable, if
$$T(\omega) = \int_{M}\langle \omega,\xi\rangle\theta \ d\mathcal{H}^n, \ \ \omega \in \mathcal{D}^n(U),$$
where $M \subset U$ is countably $n$-rectifiable, $\theta \geq 0$ is a locally $\mathcal{H}^n$-integrable integer valued function and,  for $\mathcal{H}^n$-a.e. $x \in M$, $\xi(x) = e_1\wedge\cdots\wedge e_n$, where $\{e_1,\cdots,e_n\}$ is an orthonormal basis of the approximate tangent space $T_{x}M$. In this case, we write $T = \tau(M,\theta,\xi)$. Also, we denote by $\mu_{T} = \mathcal{H}^n\mres \theta$ the Radon measure induced by the current $T$.

An $n$-varifold in $U$ is a Radon measure on $G_{n,k}(U) := U\times G(n+k,n)$, where $G(n+k,n)$ is the Grassmannian of $n$-hyperplanes in $\mathbb{R}^{n + k}$. An integer multiplicity rectifiable $n$-varifold $\mathcal{V} = v(M,\theta)$ is defined by
$$\mathcal{V}(f) = \int_{M}f(x,T_x M) \theta(x)\ d\mathcal{H}^n, \ f \in C_{c}\left(G_{n,k}(U),\mathbb{R}\right),$$
where $M \subset U$ is countably $n$-rectifiable and $\theta \geq 0$ is a locally $\mathcal{H}^n$-integrable integer valued function. In particular, given an integer multiplicity rectifiable current, forgetting the orientation we have an associated integer multiplicity rectifiable varifold. Also, for $\mathcal{V} = v(M,\theta)$ we can define the first variation $\delta \mathcal{V}$ (see \cite{Sim}[chapter 4]), and for any $C^1$-vector field $\zeta$, it holds {\it the first variation formula}
\begin{equation}\label{FVF}
\delta \mathcal{V}(\zeta) = \int_{M}\diver_M \zeta \ d\mu_{\mathcal{V}}.
\end{equation}

\subsection{Minimizing Currents with Partially Free Boundary}
\label{reg1}
Consider a compact domain $W \subset \rt$ such that $\partial W = S\cup M$, where $S$ is a compact $C^2$ surface (not necessarily connected) with boundary, $M$ is a smooth, compact mean convex surface with boundary, which intersects $S$ ortogonally along $\partial S$, and $\mathring{S}\cap\mathring M = \emptyset$ (here $\mathring A$ denotes the topological interior of $A$). Let $\gamma$ be a compact $C^2$-curve which is contained in $W$ and such that $\gamma\cap S$ is either empty or consists of a finite number of points. We shall call $\gamma$ the fixed boundary and the points of $\gamma\cap S$ by corners.

Define the class $\mathfrak{C}$ of admissible currents by
\begin{eqnarray*}
\mathfrak{C} = \{T \in \de 2(\rt); \ T \ \mbox{is integer multiplicity rectifiable},\\
\spt T \subset W \ \mbox{and is compact}, \ \mbox{and} \ \spt \bigl([[\gamma]] - \p T\bigr) \subset S\},
\end{eqnarray*}
where $[[\gamma]]$ is the current associated to $\gamma$. We want to minimize area in $\mathfrak{C}$, that is, we are looking
for $T \in \mathfrak{C}$ such that
\begin{equation}\label{var.prob}
\m(T) = \inf\{\m(\tilde{T}); \ \tilde{T} \in \mathfrak{C}\}.
\end{equation}

The existence of the fixed boundary ensures that $\mathfrak{C} \neq \emptyset$. It follows from ~\cite[$5.1.6(1)$]{Fed}, that the variational problem \eqref{var.prob} has a solution (see also \cite{Gr1}). If $T \in \mathfrak{C}$ is a solution we have
\begin{eqnarray}
\m(T) &\leq & \m(T + X) \label{min.property},\\
\spt T &\subset & W,\\
\mu_T (S) &=& 0,
\end{eqnarray}
for any integer multiplicity current $X \in \de 2(\rt)$ with compact support such that $\spt X \subset W$ and $\spt \p X \subset S$.

In order to apply the known regularity theory for $T$ we need the following results.

\begin{proposition}\label{max.princ}
If $T$ is a solution of \eqref{var.prob}, then either $\spt T\setminus\gamma \subset W \setminus M$ or $M \subset \spt T$. 
\end{proposition}

\begin{proof}
The proof follows the same ideas as in ~\cite[Lemma A.1]{Wan}, so we only describe the construction needed and refer the reader to \cite{Wan} for the details. Denote $\Sigma = \spt T$ and let $V$ be the varifold associated to $T$. Let $\zeta$ be a $C^1$ vector field with compact support on $W$ such that $\zeta(x) \in T_x S$ for any $x \in S$, and $\langle\zeta,\nu_M\rangle\geq 0$ on $M$, where $\nu_M$ is the inward pointing unit normal of $M$. By \eqref{min.property} we have
\begin{equation}
\delta V(\zeta) \geq 0.
\end{equation}

Suppose that $\Sigma$ does not contain $M$. By the main result in \cite{SolWh}, $\Sigma$ does not intersect the interior of $M$ (the result in \cite{SolWh} is stated in the case where $M$ is a minimal surface and $V$ is stationary but, as remarked at the end of the paper, the proof also works in our more general situation). 

Suppose there is a point $p \in (\partial M\setminus\gamma)\cap \si$. Let $F: M\times(-\delta,\delta) \to W$ be a diffeomorphism which is associated with an extension of $\nu_M$. By ~\cite[Lemma A.2]{Wan}, there exist $\epsilon > 0$ and a neighborhood $U \subset W$ of $p$ such that, for any $0 < s < \epsilon$ and any non negative function $w: M\cap U \to \mathbb{R}$ with $\|w\|_{C^{2,\alpha}} < \epsilon$, there exist $C^{2,\alpha}$ functions $v_t: M\cap U \to \mathbb{R}$, $t \in (-\epsilon,\epsilon)$, satisfying the following:
\begin{enumerate}
\item the graphs of $v_t$ foliate $(M\cap U)\times (-\epsilon,\epsilon)$,
\item $\textrm{graph}(v_t)$ meets $\partial W$ orthogonally along $\partial M\cap U$, $\forall \ t \in (-\epsilon,\epsilon)$,
\item $\forall \ t \in (-\epsilon,\epsilon)$, $\textrm{graph}(v_t)$ has mean curvature equal to $s$ (with respect to the downward pointing normal vector of the graph),
\item $v_t(q) = w(q) + t,$ for any $q \in \beta := \partial (M\cap U)\cap\mathring{W}$.
\end{enumerate}
In \cite{Wan}, $M$ is a minimal surface; however, the construction relies on the implicit function theorem and uses ~\cite[Appendix]{Whi} and ~\cite[Section 3]{ACS2}, with a small modification on the map needed, thus it also works if $M$ is mean-convex.

Moreover, we can choose $s$ small enough such that $v_0(p) > 0$ and choose $w$ so that if $F(q,r) \in \si$ for $q \in \beta$, then $r \geq w(q)$. Let $t_0$ be the smallest $t$ so that $v_t$ intersects $\si$. Then $t_0 < 0$ necessarily, which implies that $\si$ does not intersect $\textrm{graph}(v_{t_0}|_{\beta})$; hence, the intersection occurs at the interior or at the free boundary of $\si$. However, observe that by construction $\textrm{graph}(v_{t_0})$ has mean curvature vector pointing towards $\si$ so, by the main results of \cite{Whi2} and \cite{LiZhou}, the graph of $v_{t_0}$ can not intersect $\si$ neither at the interior nor at the free boundary, which is a contradiction. Although in \cite{LiZhou} the varifold is assumed to be stationary, the proof is by contradiction and relies on the construction of a vector field $\zeta$ as above such that $\delta V(\zeta) < 0$; hence, it also works on our case.
\end{proof}

\begin{figure}[!h]
\includegraphics[scale=0.4]{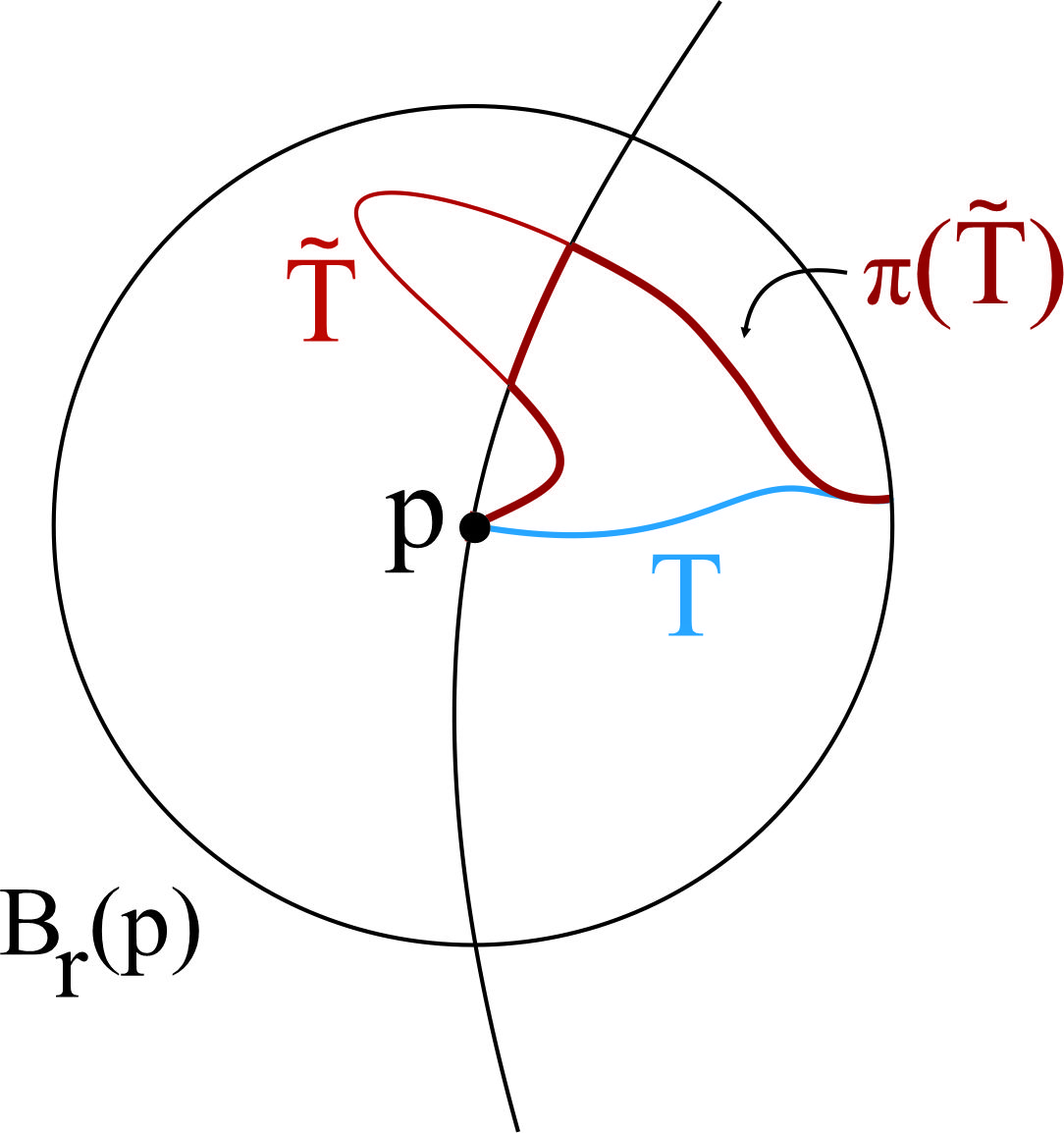}
\caption{Picture in one dimension less. Here $p\in \mbox{int}(\gamma)\cap M $.}
\label{fig-ball}
\end{figure}

\begin{proposition}\label{edelen}
Let $p\in \mbox{int}(\gamma)\cap M$. If $T$ is a solution of \eqref{var.prob}, then there exists a uniform constant $C'$ such that for $r > 0$ sufficiently small we have
$${\bf M}_{B_r(p)}(T)\leq (1+C'r){\bf M}_{B_r(p)}(T + X),$$
for any $X \in \de 2(\rt)$ such that $\spt X \subset B_r(p)$ and is compact, and $\partial X = 0$. 
\end{proposition}

\begin{proof}
To simplify the notation let us write $T$ to denote $T\cap B_r(p).$

Denote by $\pi$ the nearest point projection onto $W$. Observe that in a piece of a tubular neighborhood of $M$ containing $p$ the map $\pi$ is well defined, piecewise smooth and Lipschitz. Consider $r > 0$ such that $\pi$ is well defined in $B_r(p)$. Observe that we can find a constant $C > 0$ (independent of $r$) such that
$$
|D\pi(q)|\leq 1+Cr, \  \mbox{for a.e.} \  q\in B_r(p).
$$

Let $X$ be as in the statement of the proposition. Denote $\tilde{T} = T + X$ (see Figure \ref{fig-ball}). Then,
\begin{equation}\label{ineq.alm.min}
{\bf M}(T)\leq {\bf M}(\pi(\tilde T)) \leq (1+Cr)^2 {\bf M}(\tilde T) \leq (1+C'r){\bf M}(\tilde T),
\end{equation}
for some constant $C'$. So the proof is complete.
\end{proof}

We then have the following regularity result.

\begin{theorem}
Let $T$ be a solution of \eqref{var.prob}. Then, away from the corners, $T$ is supported in a connected oriented embedded minimal $C^2$-surface, which meets $S$ orthogonally along $\spt ([[\gamma]] - \p T\bigr)$. 
\label{thm.reg1}
\end{theorem}

\begin{proof}
From the classical interior regularity theory developed by DeGiorgi (here $n = 2 < 7$), see \cite{Fed},  we know that in a neighborhood of each $x \in \spt T \setminus \spt \p T$, T is given by ($m$-times, $m \in \na$) integration over an embedded minimal surface. 
The regularity near a point $x$ at the fixed part of the boundary away from the corners follows from the work of Hardt and Simon \cite{HarSim} on the case $x \in \mathring W$, and for the case $x \in M$ we can use, by Proposition \ref{edelen} and Remarks 0.2 and 0.3 in \cite{Ed}, the results in \cite{Ed} (let us remark that in the proof of Theorem \ref{thm.reg2} we will get to this same situation after a reflection and we will give more details on how to use the results in \cite{Ed}). Since by Proposition \ref{max.princ} the free part of the boundary is contained in $S\setminus \p S$, we can use the result by Gr\"uter \cite{Gr2} to conclude the regularity at the free boundary (away from the corners). Therefore, away from the corners, $T$ is supported in a connected oriented embedded minimal $C^2$-surface, which meets $S$ orthogonally along $\spt ([[\gamma]] - \p T\bigr)$.
\end{proof}


It remains the question about the regularity of $T$ at the corners. In \cite{Gr3}, Gr\"uter reports joint work with L. Simon (unpublished) where they would prove regularity in a similar situation. We develop here the ideas present in \cite{Gr3} to prove the regularity for the case where $\gamma$ meets $S$ orthogonally. 

Arguing as in Section $3$ of \cite{Gr2}, we can reduce the problem of local regularity at a corner to the following situation. Applying a translation and a dilation if necessary we can suppose one of the corners is located at the origin $\ze \in \rt$ and the open ball $B_{3}(\ze)$  (centered at the origin with radius 3) is decomposed by $\partial W = M\cup S$ into two open $3$-cells, that is
\begin{equation}\label{cells}
B_3(\ze) = B_3^{-}\cup\bigl(\partial W\cap B_3(\ze)\bigr)\cup B_3^{+},
\end{equation}
where $B_3^{-}$ and $B_3^{+}$ are homeomorphic to the $3$-dimensional unit ball and the decomposition is disjoint. 

Consider a rectifiable  $T \in \de 2\bigl(B_{3}(\ze)\bigr)$ of integer multiplicity satisfying: 
\begin{eqnarray}
\spt T \subset \overline {B_3^{+}}, \ \ze \in \spt T,\\
\spt \bigl([[\gamma]] - \p T\bigr) \subset S, \\
\m(T) < +\infty, \\
\m _U(T) \leq \m _U(T + X), \label{min}
\end{eqnarray}
for every open set $U \subset\subset B_{3}(\ze)$ and for any integer multiplicity current $X \in \de 2\bigl(B_{3}(\ze)\bigr)$ such that $\spt X \subset U\cap W$ and $\spt \p X \subset S$. It also holds $\mu_T (S\cap U) = 0$, for every open set $U \subset\subset B_{3}(\ze)$.

\subsection{Regularity at the corner}
\label{reg2}

Extend $S$ to a closed smooth surface $\widetilde{S} \subset \mathbb{R}^3$ such that a tubular neighborhood of $\widetilde{S}$ contains $B_3(\ze)$ (applying a dilation if necessary). Define the reflection $\Phi: B_3(\ze) \to \rt$ across $\widetilde{S}$ by
$$\Phi(y) = 2\Pi(y) - y,$$
where $\Pi(y)$ is defined as the unique point in $\tilde S$ such that $\dis (y,\tilde S) = |y - \Pi(y)|$. Since our questions are local, we can ensure that $\Pi$ is well defined (after a dilation if necessary) and continuously differentiable. Geometrically we can see $\Phi$ as follows: the line through $x = \Pi(y)$ with direction $\xi(x)$ (a unit normal vector to $\tilde S$ at $x$) is parametrized by $t\mapsto x + t\xi(x)$, so if $y = x + t\xi(x)$, we have $\Phi(y) = x - t\xi(x)$. It is easy to see that $\Phi^2 = \id$.

Define $T' = T - \Phi_{\#}(T)$. Thus $T' \in \de 2(\rt)$ has integer multiplicity, $\spt \p T' \subset \gamma\cup\Phi(\gamma)$ and
\begin{equation}\label{V}
\spt T' \subset B_3^{+}\cup\bigl(\partial W\cap B_3(\ze)\bigr)\cup \Phi(B_3^{+}).
\end{equation}
Moreover, $\m(T') < +\infty$. Denote $B = B_1(\ze)$ and $T'\mres B$ by $\ti$. Hence, $\ti \in \de 2(B)$ has integer multiplicity, finite mass and the support of its boundary is contained in $\gamma\cup\Phi(\gamma)$. 

Now we will prove that $\ti$ is regular at $\ze$ and, of course, this implies the regularity of $T$ at $\ze$. For this purpose, we will first adapt some ideas of \cite{Gr2} to show that $\ti$ has a tangent cone at $\ze$ which is area-minimizing.

\begin{lemma}\label{der}
Consider $y = x + r\xi(x)$, where $x \in \widetilde{S}$ and $\xi(x)$ is a unit normal vector to $\tilde S$ at $x$. Then, for $r$ small enough, the derivative of $\Phi$ satisfies
\begin{equation}
1 - c_1 r \leq |D\Phi(y)| \leq 1 + c_1 r,
\end{equation}
where $c_1$ is a positive constant.
\end{lemma}

\begin{proof}

Given $x \in \tilde S$ and $ v\in T_x \tilde S$, consider a curve $\alpha : (-\epsilon,\epsilon) \to \tilde S$ such that $\alpha(0) = x$ and $\alpha'(0) = v$. Then
\begin{eqnarray*}
D\Phi(x)\cdot v &=& \frac{d}{ds}\biggl|_{s=0} \Phi\bigl(\alpha(s)\bigr) = \frac{d}{ds}\biggl|_{s=0} \alpha(s) = v,\\\\
D\Phi(x)\cdot \xi &=& \frac{d}{ds}\biggl|_{s=0} \Phi\bigl(x + s\xi(x)\bigr) =  \frac{d}{ds}\biggl|_{s=0} \bigl(x - s\xi(x)\bigr) = -\xi(x),
\end{eqnarray*}

\begin{eqnarray*}
D^2\Phi(x)\cdot (\xi,\xi) &=& \frac{d^2}{ds^2}\biggl|_{s=0} \Phi\bigl(x + s\xi(x)\bigr) = \frac{d^2}{ds^2}\biggl|_{s=0} \bigl(x - s\xi(x)\bigr) = 0,\\\\
D^2\Phi(x)\cdot (\xi,v) &=& \frac{\p^2}{\p t\p s}\biggl|_{t=0,s=0} \Phi\bigl(\alpha(s) + t\xi(\alpha(s))\bigr) = \nabla^{\rt}_v \xi.
\end{eqnarray*}

\noindent
So, $|D\Phi(x)| = 1$, and by Taylor's theorem we have for $r$ small enough
$$\bigl|D\Phi\bigl(x + r\xi(x)\bigr)\bigr| = 1 + \bigl\langle D\Phi(x),D^2\Phi(x)\cdot(\xi,\cdot)\bigr\rangle r + O(r),$$
where $|O(r)| \leq  \lambda r$, for some constant $\lambda > 0$. By the computation above,
$$\bigl|\bigl\langle D\Phi(x),D^2\Phi(x)\cdot(\xi,\cdot)\bigr\rangle\bigr| \leq \kappa,$$
where $\kappa$ is the supremum of the norm of the second fundamental form of $S$. Therefore, the result follows.
\end{proof}

\begin{lemma}\label{dens}
There exists $\Theta(\mu_{\ti},\ze) := \displaystyle \lim_{r \to 0} \frac{\mu_{\ti}\bigl(B_{r}(\ze)\bigr)}{\pi r^2}$.
\end{lemma}

\begin{proof}
The surface $\widetilde{S}$ separates the ball $B$ in two connected components $\mathcal{U}$ and $\mathcal{U}'$. Suppose (without loss of generality) $\mathring M \subset \mathcal{U}$. Let $V$ be an open set such that $V \subset\subset  B_{r}(y_0) \subset B,$
for some $y_0 \in B\cap S$ . Let $Y \in \de 2(B)$ be of integer multiplicity such that $\spt Y \subset V\cap\mathcal{U}$ and $\spt\p Y \subset \widetilde{S}$. Denote by $\pi: \mathcal{U} \to B_3(\ze)\cap W$ the nearest point projection onto $W$. Arguing as in Proposition \ref{edelen} we conclude that there is a constant $c_2$ such that
\begin{equation}\label{eq:alm.min.2}
{\bf M}_{V}(T) \leq {\bf M}_{V}\bigl(\pi(T + Y)\bigr) \leq (1+c_2r){\bf M}_{V}(T + Y).
\end{equation}

Now let $U$ be an open set such that $U = \Phi(U)$ and $U \subset\subset B_{R}(x_0) \subset B,$
for some $x_0 \in S\cap B$. Let $X \in \de 2(B)$ be of integer multiplicity such that $\spt X \subset U$ and $\p X=0$. We can write 
$$X = X_1 + X_2= X\mres\mathcal{U} + X\mres\mathcal{U}',$$
where $\spt \p X_i \subset \widetilde{S}$, $i = 1,2$, since $\p X=0$. By \eqref{eq:alm.min.2}, for $R$ sufficiently small we have
\begin{eqnarray*}
\m_{U}(\ti + X) &=& \m_{U}(T + X_1) + \m_{U}(\Phi_{\#}T - X_2)\\
&\geq & \frac{1}{1+ c_2 R}\m_{U}(T) + \m_{U}\bigl(\Phi_{\#}(T - \Phi_{\#}X_2)\bigr)\\
&\geq & \frac{1}{1+ c_2 R}\m_{U}(T) + (1 - c_1 R)^2\m_{\Phi(U)}(T - \Phi_{\#}X_2)\\
&\geq & \frac{1}{1+ c_2 R}\m_{U}(T) + \frac{(1 - c_1 R)^2}{1+ c_2 R}\m_{\Phi(U)}(T)\\
&\geq & \frac{1}{1+ c_2 R}\m_{U}(T) + \frac{(1 - c_1 R)^2}{1+ c_2 R}\m_{U}(\Phi_{\#}T)\\
&\geq & \frac{1}{1+ c_2 R}\m_{U}(T) + \frac{1 - c_3 R}{1+ c_2 R}\m_{U}(\Phi_{\#}T)\\
&=& \frac{1}{1+ c_2 R}\m_{U}(\ti) - \frac{c_3 R}{1+ c_2 R} \m_{U}(\Phi_{\#}T)\\
&\geq & \frac{1}{1+ c_2 R}\m_{U}(\ti) - \frac{c_3 R}{1+ c_2 R} \m_{U}(\ti)\\
&\geq & (1 - c_4 R)\m_{U}(\ti) - c_5 R \m_{U}(\ti).
\end{eqnarray*}

Thus
\begin{equation}\label{ineq.mass}
\m_{U}(\ti) \leq \m_{U}(\ti + X) + c_6 R\m_{U}(\ti).
\end{equation}

For $r>0$ sufficiently small, denote by ${B_r}$ the closed ball centered at the origin $\ze$ with radius $r$. Consider the curve $\alpha_r = \spt\left(\partial (\widetilde{T}\mres {B_r})\right)\setminus\left[\bigl(\gamma\cup\Phi(\gamma)\bigr)\cap {B_r}\right]$ and take the current $C_r=[[\ze \#\alpha_r]]+[[E_r]],$ where $\ze \#\alpha_r$ is the cone over $\alpha_r$, and $E_r\subset {B_r}$ is the strip bounded by $\gamma\cap {B_r}$ and the two radii joining $\ze$ to the endpoints of $\gamma\cap {B_r}$ (see Figure \ref{fig.Er}). Observe that $\m_{B_r}\bigl([[E_r]]\bigr)\leq c_7r^3,$ for some constant $c_7,$ and by the coarea formula 
$$\m_{B_r}\bigl([[\ze \#\alpha_r]]\bigr)=\displaystyle\int_0^r\frac{t}{r}\ell(\alpha_r)\ dt=\frac{r}{2}\ell(\alpha_r).$$

\begin{figure}[!h]
\includegraphics[scale=0.5]{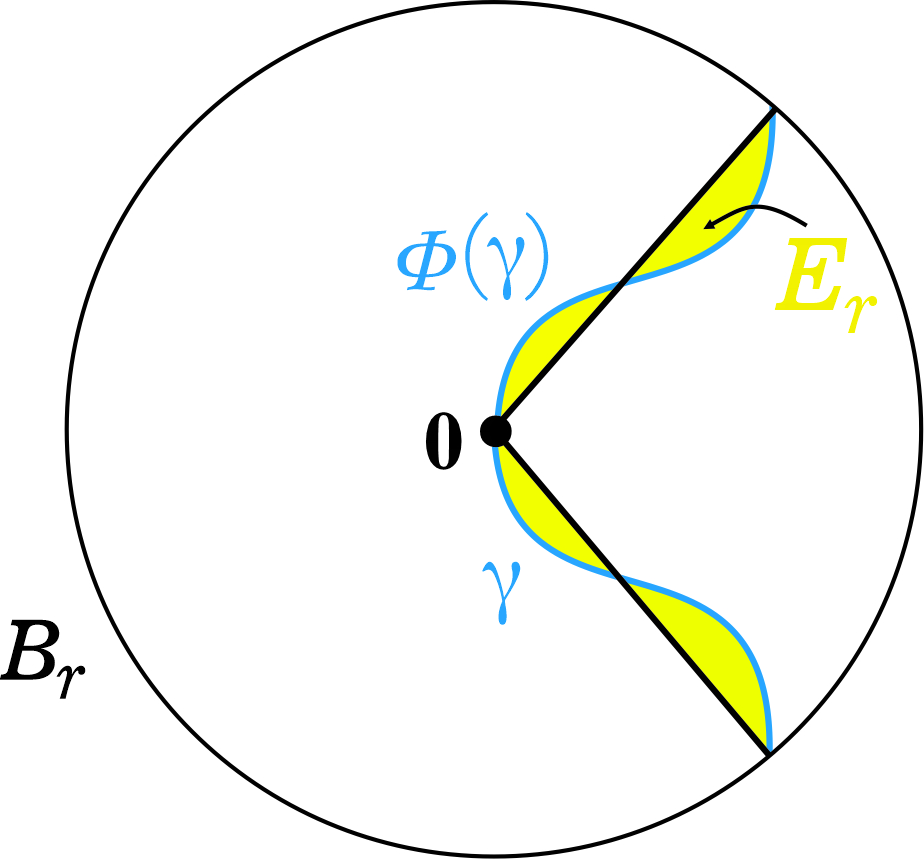}
\caption{The strip $E_r$.}
\label{fig.Er}
\end{figure}

To simplify our notation, let us denote $M(r)=\m_{B_r}(\widetilde{T}).$ Observe that $\partial (\widetilde{T}\mres {B_r})=\p C_r$, so by $(\ref{ineq.mass})$, we have
$$
M(r)\leq(1+c_8r)\m_{B_r}(C_r),
$$
for some constant $c_8$. Hence,
\begin{equation}
M(r)\leq (1+c_8r)(\m_{B_r}([[\ze\#\alpha_r]])+\m_{B_r}([[E_r]]))\leq (1+c_8r)\left(\frac{r}{2}\ell(\alpha_r)+c_7 r^3\right).
\label{eq.11}
\end{equation}

Define $\si = \spt(\widetilde{T}\mres B_r)$. By the coarea formula
$$
M(r)=\int_0^r\int_{\alpha_t}\frac{1}{|\nabla_{\si} r|}\ ds\ dt,
$$
thus
\begin{equation}
\label{eq.12}
M'(r)\geq \ell(\alpha_r), \ \ \mbox{for a.e.} \ r.
\end{equation}

Applying (\ref{eq.12}) in (\ref{eq.11}), we get
$$
M(r)\leq (1+c_8r)\left(\frac{r}{2}M'(r)+c_7r^3\right),
$$
which implies that for some constant $c$ and a.e. $r$
\begin{equation}
\label{eq.13}
M'(r)\geq \frac{2}{r}\left[(1-cr)M(r)-cr^3\right].
\end{equation}

Define $h(r)=f(r)\frac{M(r)}{r^2},$ where $f$ is a differentiable function which will be chosen later. We have for a.e. $r$
\begin{eqnarray*}
h'(r)&=&f'(r)\frac{M(r)}{r^2}-2\frac{f(r)}{r^3}M(r)+f(r)\frac{M'(r)}{r^2}\\
&\geq& f'(r)\frac{M(r)}{r^2}-2\frac{f(r)}{r^3}M(r)+\frac{f(r)}{r^2} \frac{2}{r}\left[(1-cr)M(r)-cr^3\right]\\
&=&\frac{f(r)M(r)}{r^2}\left(\frac{f'(r)}{f(r)}-2c\right)-2c f(r).
\end{eqnarray*}

Taking $f(r)=e^{2cr}$, we have $\frac{f'(r)}{f(r)}=2c$ and $\left(h(r)+e^{2cr}\right)'\geq 0$ for a.e. $r.$ In particular, the function $e^{2cr}\frac{M(r)}{r^2}+e^{2cr}-1$ is non-decreasing and, therefore, 
$\lim_{r\to 0}\frac{M(r)}{r^2}$ exists.
\end{proof}

\begin{definition}
Consider the map defined by $\eta_{x_0,\lambda} = \lambda^{-1}(x - x_0)$, $x \in \R^{n+1}$. If $x_0 = \ze$, we simply write $\eta_{\lambda}$. Suppose $T \in \de n(U)$ is integer multiplicity and $x_0 \in \spt T$. If there exist a sequence $\{\lambda_j\}_{j \in \na}$ converging to $0$ and an integer multiplicity current $C \in \de n(\R^{n+1})$ such that
$$(\eta_{x_0,\lambda_j})_{\#}T \to C, \  \ \mbox{and} \  \ (\eta_{x_0,\lambda})_{\#}C = C, \ \forall \ \lambda > 0,$$
we call $C$ an oriented tangent cone to $T$ at $x_0$.
\end{definition}

\begin{theorem}\label{cone.exist}
There is an oriented tangent cone $C$ to $\ti$ at $\ze$ such that $\spt\p C$ is an oriented straight line. Moreover, 
\begin{enumerate}
\item $C$ is minimizing in $\rt$, that is, for any open set $U \subset\subset \rt$ and any integer multiplicity current $X \in \de 2(\rt)$ satisfying $\spt X \subset U$ and $\p X=0$, we have
$$\m_U(C) \leq \m_U(C + X);$$
\item if $T_j = (\eta_{\lambda_j})_{\#} \ti \to C$ we have
$$\mu_{T_j} \to \mu_{C}, \ \ \mbox{and} \ \ \Theta(C,\ze) = \Theta(\mu_{\ti},\ze).$$
\end{enumerate}
\end{theorem}

\begin{proof}
Consider a sequence $\{\lambda_j\}_{j \in \na}$ converging to $0$. By Lemma \ref{dens}, we have 
$$\m_{B_r(\ze)}\bigl((\eta_{\lambda_j})_{\#} \ti\bigr) = \lambda_j^{-2}\m_{B_{r\lambda_j}(\ze)}(\ti) = \lambda_j^{-2}\mu_{\ti}\bigl(B_{r\lambda_j}(\ze)\bigr) \leq c r^2.$$

Since $\gamma$ meets $S$ orthogonally at $\ze$, we have that $\gamma\cup\Phi(\gamma)$ is a $C^{1,1}$-curve. Indeed, it is clearly $C^1$, and since $\gamma$ is $C^2$, its derivative is Lipschitz, so this also holds for the duplicate curve. Hence, there exists
$$\lim_{r \to 0} \frac{\mu_{\p\ti}\bigl(B_{r}(\ze)\bigr)}{r} < +\infty,$$
therefore,
$$\m_{B_r(\ze)}\Bigl(\p\bigl((\eta_{\lambda_j})_{\#} \ti\bigr)\Bigr) = \m_{B_r(\ze)}\bigl((\eta_{\lambda_j})_{\#} \p\ti\bigr) = \lambda_j^{-1}\m_{B_{r\lambda_j}(\ze)}(\p\ti) \leq \tilde c r.$$

Thus, from the compactness theorem for integer multiplicity currents (see ~\cite[4.2.17]{Fed}), a subsequence of $\{(\eta_{\lambda_j})_{\#} \ti\}$ converges to a current $C \in \de 2(\rt)$ and the boundaries $\p\bigl( (\eta_{\lambda_j})_{\#} \ti\bigr)$ converge to $\p C \in \de 1(\rt)$ (in the subsequence). Furthermore, since $\gamma\cup\Phi(\gamma)$ is a $C^{1,1}$-curve, it has an oriented tangent line $\Gamma$ at $\ze$. In particular, $\spt\p C = \Gamma$. At the end of this proof we will show that $C$ is a cone.

Let $U \subset\subset \rt$ be an open set and fix $r > 0$ so that $U\cup\Phi(U) \subset\subset B_r(\ze)$. Let $X \in \de 2(\rt)$ be integer multiplicity such that $\spt X \subset U$ and $\p X=0$. Choose a sequence $\{\lambda_j\}$ with $\lambda_j \to 0$ and
$T_j := (\eta_{\lambda_j})_{\#}\ti \to C.$
Define $X_j = (\eta_{\lambda_j})^{-1}_{\#}X$. We have $ (\eta_{\lambda_j})^{-1}(U) \subset B_{\lambda_j r}(\ze)$; hence, for $j$ large enough, $\lambda_j r\leq 3$ and  $ (\eta_{\lambda_j})^{-1}(U) \subset B$. Then, it follows from \eqref{ineq.mass} that
\begin{eqnarray}\label{seqmass}
\m_{U}(T_j) \ = \ \m_{U}\bigl((\eta_{\lambda_j})_{\#}\ti\bigr) &=& \lambda_{j}^{-2}\m_{(\eta_{\lambda_j})^{-1}(U)}(\ti) \nonumber\\
&\leq & \lambda_{j}^{-2}\bigl[\m_{(\eta_{\lambda_j})^{-1}(U)}(\ti + X_j) + c_6\lambda_j r\m_{(\eta_{\lambda_j})^{-1}(U)}(\ti)\bigr] \nonumber\\
&=& \m_{U}(T_j + X) + c_6 \lambda_jr\m_{U}(T_j) \nonumber\\
&\leq & \m_{U}(T_j + X) + \tilde c \lambda_j r^3,\label{AAM}
\end{eqnarray}
where we used Lemma \ref{dens} in the last inequality.

Since $T_j \to C$, $\lambda_j \to 0$ and  because of the lower-semi-continuity of mass, a standard argument using \eqref{seqmass} yields that 
$$\m_{U}(C) \leq \m_{U}(C + X).$$

Next we will prove $\mu_{T_j} \to \mu_{C}$ in the sense of Radon measures. Consider a compact set $K \subset \rt$ and an open set $U \subset\subset \rt$ containing $K$. For any $\epsilon>0$, define $\phi_{\epsilon}: U \to [0,1]$ such that $\phi_{\epsilon} \equiv 1$ in some neighborhood of $K$ and $$\spt \phi_{\epsilon} \subset \{x \in \rt; \ \dis(x,K) \leq \epsilon\}.$$ 

Consider
$$U_{\alpha,\epsilon} = \{x \in \rt; \phi_{\epsilon} > \alpha\};$$
hence, for $0 \leq \alpha < 1$, we have
$K \subset U_{\alpha,\epsilon} \subset\subset \rt$.

Using \eqref{seqmass} and proceeding as in the proof of ~\cite[Theorem 34.5]{Sim}, we conclude that $\m_{U_{\alpha,\epsilon}}(T_j) \leq \m_{U_{\alpha,\epsilon}}(C) + \epsilon_j$, where $\epsilon_j \to 0$. So,
$$\lim\sup \m_{U_{\alpha,\epsilon}}(T_j) \leq \m_{U_{\alpha,\epsilon}}(C).$$
Since $K \subset U_{\alpha,\epsilon} \subset \{x; \ \dis(x,K) < \epsilon\}$ (by construction) we obtain
$$\lim\sup \mu_{T_j}(K) \leq \m_{\{x; \ \dis(x,K) < \epsilon\}}(C).$$
Hence, letting $\epsilon \to 0$, it follows that
\begin{equation}\label{limsup}
\lim\sup \mu_{T_j}(K) \leq \mu_{C}(K).
\end{equation}

By the lower semi-continuity of mass with respect to weak convergence,
we have
\begin{equation}\label{liminf}
\mu_{C}(K) \leq \lim\inf \mu_{T_j}(K).
\end{equation}

Since \eqref{limsup} and \eqref{liminf} hold for arbitrary compact $K$ and open $U \subset\subset \rt$, it follows by a standard approximation argument that 
\begin{equation}\label{radonconver}
\mu_{T_j} \to \mu_{C}.
\end{equation}

Finally, choose $r > 0$ such that $\mu_{C}\bigl(\p B_r(\ze)\bigr) = 0$ (which is true except possibly for countably many $r$). Then, \eqref{radonconver} implies that
$$\frac{\mu_{C}\bigl(B_r(\ze)\bigr)}{\pi r^2} = \lim_{j \to \infty}\frac{\mu_{T_j}\bigl(B_r(\ze)\bigr)}{\pi r^2} = \Theta(\mu_{\ti},\ze).$$
Thus
\begin{equation}\label{dens.cone}
\Theta(\mu_{C},\ze) = \frac{\mu_{C}\bigl(B_r(\ze)\bigr)}{\pi r^2} = \Theta(\mu_{\ti},\ze).
\end{equation}

Finally, since $C$ is area-minimizing, it is a stationary varifold. So, using \eqref{dens.cone} in the monotonicity identity at the boundary (see  ~\cite[equation 31]{Bour}) for $0 < s < r$ we have
\begin{equation}\label{monot}
0 = \frac{\mu_{C}(B_r)}{r^2} - \frac{\mu_{C}(B_s)}{s^2} = \int_{B_r\setminus B_s} \frac{\bigl|\nabla^{\perp}(|x|)\bigr|^2}{|x|^2}\ d\mu_{C} + \int_{s}^{r}t^{-3}\int_{B_t\cap\Gamma}\langle x,\eta\rangle \ d\mu_{\p C}\ dt,
\end{equation}
where $B_t = B_t(\ze)$ and $\eta$ is a conormal to $\Gamma = \spt\partial C$. Since $\Gamma$ is a line containing $\ze$, the second integral on \eqref{monot} vanishes. So, we can argue as in ~\cite[Theorem $19.3$]{Sim} to conclude that $C$ is a cone.
\end{proof}

\begin{theorem}
There is a neighborhood $U$ of $\ze$ such that $U\cap\spt \ti$ is an embedded oriented $C^{1,\alpha}$-surface with boundary, for $0 < \alpha < 1$.
\label{thm.reg2}
\end{theorem}

\begin{proof}
Since $M$ meets $S$ orthogonally, the set $\widetilde{M} = \bigl(M\cap B_3(\ze)\bigr)\cup\Phi\bigl(M\cap B_3(\ze)\bigr)$ is a $C^{1,1}$-surface; hence, it satisfies an exterior ball condition at $\ze$ (there exists a ball $B_r(y) \subset \mathbb{R}^3\setminus W\cup\Phi(W)$ which is tangent to $\widetilde M$ at $\ze$). Also, $\spt\widetilde{T} \subset \bigl(W\cap B_3(\ze)\bigr)\cup\Phi\bigl(W\cap B_3(\ze)\bigr)$. From Proposition $3$ and equation \eqref{ineq.mass}, there is a constant $c$ such that
\begin{equation}\label{almostminimizing}
{\bf M}_{U}(\widetilde T)\leq (1+cr){\bf M}_{U}(\widetilde T + X),
\end{equation}
for any $U \subset\subset B_{R}(x) \subset B_3(\ze)$, and any $X \in \de 2(\rt)$ such that $\spt X \subset U$ and $\partial X = 0$.

Now we can argue as in ~\cite[Lemma 1.1]{Ed} (we will only sketch the arguments and refer the reader to \cite{Ed} for the details). Denote $\mathcal{W} = \mathbb{R}^3\setminus W\cup\Phi(W)$. One can construct a (topological) surface with boundary, denoted by $\si$, which is contained in the closure of $\mathcal{W}\cap B_3(\ze)$ and such that $\partial \si = \spt(\partial\widetilde{T}\mres B_3(\ze))$. Observe that $\partial\left(\widetilde{T} - [[\si]]\right)\mres B_3(\ze) = 0$. Then, using the decomposition of codimension one currents ~\cite[Theorem 27.6]{Sim} and arguing as in \cite{Ed}, we find a family of open sets $\{E_{j}\}_{j\in \mathbb{Z}}$, where $E_{j+1} \subset E_{j} \subset B_3(\ze)$, $\forall j \in \mathbb{Z}$, and such that for some $j_*$ we have 
\begin{eqnarray}
\widetilde{T} - [[\si]] = \sum_{j=-\infty}^{\infty} \partial[[E_j]], \ \ 
{\bf M}(\widetilde{T}) = {\bf M}\bigl(\partial[[E_{j_*}]] + [[\si]]\bigr) + \sum_{j \neq j_*}{\bf M}\bigl(\partial[[E_j]]\bigr),\label{eq:minimal}\\
\partial[[E_{j_*}]]\mres \mathcal{W} = - [[\si]], \ \ \partial[[E_j]] \mres \mathcal{W} = 0, j \neq j_*,\label{eq:ext.dom}
\end{eqnarray}
and $\partial[[E_{j_*}]] + [[\si]]$ and each $\partial[[E_{j\neq j_*}]]$ satisfy an almost-minimizing property like \eqref{almostminimizing}. 

Consider $\partial[[E_j]]$, $j\neq j_{*}$, such that its support contains $\ze$ and denote by $C_j$ its tangent cone at $\ze$. Also, denote by $C_{j_*}$ the tangent cone of $\partial[[E_{j_*}]] + [[\si]]$ at $\ze$. The existence of the cones and the fact that each $C_j$ is area-minimizing follows as in Theorem \ref{cone.exist}. Also, there exist open sets $F_j$, a half-space $\mathcal{H}$ and a half-plane $H \subset \mathcal{H}$ such that $C_{j_*} = \partial[[F_{j_*}]] + [[H]]$, $\spt C_{j_{*}} \subset \mathcal{H}$, $C_{j} = \partial[[F_j]]$ and $F_j \subset \mathcal{H}$ for $j\neq j_*$. Thus $C_{j_*}$ is a multiplicity one half-plane and each $C_j$ must be a multiplicity one plane for $j \neq j_*$. In particular, $\Theta\left(C_{j_*},\ze\right) = 1/2$ and $\Theta\left(C_{j},\ze\right) = 1$, for $j \neq j_*$. Since $\gamma\cup\Phi(\gamma)$ is a $C^{1,1}$-curve, we can use the results of \cite{DS} to conclude that $\partial[[E_{j_*}]] + [[\si]]$ and each $\partial[[E_{j\neq j_*}]]$ are $C^{1,\alpha}$-surfaces, for $0 < \alpha < 1$, in some neighborhood of $\ze$.

By Proposition \ref{max.princ}, either $\spt\widetilde{T} \subset M\cup\Phi(M)$ or $\spt\tilde{T}\cap (M\cup\Phi(M)) = \gamma\cup\Phi(\gamma)$. In the first case we already have regularity, so suppose the second case happens. Our final step is to show that each $\partial E_{j\neq j_*}$ is actually an empty set. Fix $j,$ for $ j\neq j_*$. By (\ref{eq:ext.dom}), we know that $\partial E_j$ is on one side of $M$. Since $C_j$ is a multiplicity one plane, then it is necessarily the tangent plane to $M$ at $\ze$ (otherwise there would be points of $\partial E_j$ in $\mathcal W$). Hence, since $M$ and $\p E_j\cap W$ are tangent at $\ze$, they can be written locally as graphs over closed regions ${\Omega}_1$ and ${\Omega}_2$ of their tangent plane $T_{\ze}M$ (with possibly ${\Omega}_1\neq {\Omega}_2$); that is, there exists a closed neighborhood $U$ of $\ze$ in $B_3(\ze)\cap W$ such that $M\cap U=\mbox{Graph}(u_1)$ and $\partial E_j\cap U=\mbox{Graph}(u_2)$, where $u_1:\Omega_1\to \mathbb{R}$ and $u_2:\Omega_2\to \mathbb{R}$ with $u_1(0,0)=u_2(0,0)=0$. Observe that since $\partial E_j$ is tangent to $M$ at $\ze$ and has no boundary, the point $\ze$ is not a corner for $\p E_j\cap W$, just a point in its boundary.

Let $\gamma_1$ and $\gamma_2$ be two arcs of the curves $M\cap S$ and $\partial E_j\cap S,$ respectively, that contain the origin $\ze$. We can see the arc $\gamma_i$ as the graph over a curve $\alpha_i$ in $\partial \Omega_i, i=1,2.$ Furthermore, as everything is local, we can see $\alpha_i$ as the graph of a $C^{1, \alpha}$ function $f_i:I\subset \rr\to\rr$ with $f_i(0)=0, i=1,2.$ Hence, we have $\gamma_i(t)=(t, f_i(t), u_i(t,f_i(t))), t\in I,$ and we can assume that $\Omega_i\subset \{(x,y); x\in I \ \mbox{and} \ y\leq f_i(x)\}$ (see Figure \ref{fig-omega}). Observe that since $M$ is tangent to $\p E_j$ at $\ze$ and transversal to $\s^2$, the curves $\gamma_1$ and $\gamma_2$ are tangent at $\ze;$ in particular, we can assume that $f_1'(0)=f_2'(0)=0.$

\begin{figure}[!h]
\includegraphics[scale=1.3]{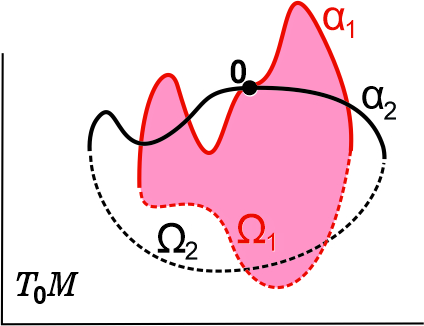}
\caption{Closed regions $\Omega_1$ and $\Omega_2$.}
\label{fig-omega}
\end{figure}

Since $f_i$ is $C^{1,\alpha},$ there exists a constant $C_i>0$ such that 
$$
|f_i'(t)-f_i'(0)|\leq C_it^{\alpha},
$$
which implies
$$
f_i'(t)\geq -C_it^\alpha.
$$
Hence,
$$
f_i(x)=f_i(0)+\int_0^x f_i'(t)dt \geq -C_ix^{\alpha+1}.
$$
Taking $C=\mbox{max}\{C_1,C_2\}$, we get 
$$
f_i(x)\geq -Cx^{\alpha+1}, \ \ i=1,2.
$$
Call $h(x)=-Cx^{\alpha+1}$. We have that $h$ is a $C^{1,\alpha}$ function with $h(0)=0, h'(0)=0$ and $\mbox{min}\{f_1(x),f_2(x)\}\geq h(x).$ Therefore, we can find a domain $\Omega\subset \Omega_1\cap \Omega_2$  containing the origin $\ze$ with $C^{1,\alpha}$ boundary which contains the graph of $h$ (see Figure \ref{fig-omega2}). In particular, $\partial \Omega$ has a Dini continuous normal and a Hopf-type lemma can be applied to $\Omega$ (see, for instance, ~\cite[Theorem 1.3]{Sa}, or the notes at pg. 46 of \cite{GT}). Since $M$ is mean convex, $\partial E_j\cap W$ is minimal (by the regularity of the original current $T$, Theorem \ref{thm.reg1}) and $\partial E_j\cap W$ is on one side of $M$, we can apply the Hopf Lemma to $(u_1 - u_2)|_{\Omega}$ and conclude that its normal derivative at the origin has a strict sign. However, this is not true once $M$ and $\partial E_j$ are tangent at $\ze.$ Thus $\partial E_j$ is empty.

\begin{figure}[!h]
\includegraphics[scale=1.3]{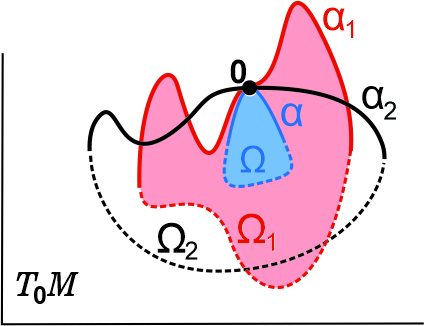}
\caption{Here $\alpha$ represents the graph of $h$.}
\label{fig-omega2}
\end{figure}
\end{proof}

Therefore, combining all the regularity results from Sections \ref{reg1} and \ref{reg2}, we get the following theorem.

\begin{theorem}
\label{thm.last}
Let $T$ be a solution of \eqref{var.prob}, where $\gamma$ is a $C^2$ curve. Then $T$ is supported in a connected oriented embedded minimal surface, which meets $S$ orthogonally along $\spt\bigl([[\gamma]] - \p T\bigr)$ and it is $C^{1,\alpha}$, $0<\alpha<1$, in a neighborhood of $\gamma\cap S$ and $C^2$ away from $\gamma\cap S$.
\end{theorem}

\begin{remark}
\label{last}
If $\gamma$ has isolated singularities (for example, $n$-prong singularities) at the interior or the boundary, the same regularity result holds true for points away from these singularities.
\end{remark}

\appendix

\section{A maximum principle at a corner}\label{app}




Here we are going to show a kind of a maximum principle at a corner that we use in the proof of Theorem \ref{thm-nul}. 

As in Section \ref{sec.main}, given a region $W\subset \ba^3$, we will denote by $\partial W$ the closure of the part of the boundary of $W$ contained in int$\ba^3$.  

\begin{theorem}
\label{thm.corner}
Let $W\subset \ba^3$ be a connected closed region with (non-strictly) mean convex boundary  such that $\partial W$ meets $\s^2$ orthogonally along its boundary and $\partial W$ is smooth. Suppose there is a closed half-disk $D^{+}$ contained in $W$ and denote by $p$ one of its corners. If $\partial W$ is tangent to $D^+$ at $p,$ then $W$ is necessarily a halfball. 
\end{theorem}

\begin{proof}
Let $P$ be the plane containing $D^{+}$. We are assuming that $\partial W$ and $D^+$ are tangent at $p.$ Observe that, by the classical (interior or free boundary version of the) maximum principle, we already know that either $W$ is a halfball or $\partial W\cap D^+$ is contained in the diameter of $D^+$ (see the proof of Theorem \ref{thm-nul} for more details). 

We can choose coordinates $(x,y,z)$ so that $\partial_x$ is tangent to the circumference at $p$ in $D^+$ and belongs to the halfplane containing $D^+$, $\partial_y$ is in the direction of its diameter and coincides with the outward conormal vector to $D^{+}$, and $\partial_z$ is the normal vector to $P$ pointing towards $\partial W$ at the points $B_\delta(p)\cap D^+$.  

Since $\partial W$ is locally on one side of $D^+$, if we define $M$ as $\partial W\cap B_{\delta}(p),$ we can take $\delta>0$ sufficiently small so that $M$ can be written as a graph over a region in $P$ containing $p$. Consider $\pi: \rr^3\to P$ the orthogonal projection onto $P.$ We are choosing the unit normal vector $n$ to $P$ pointing into $M,$ that is, $n=\partial_z$. Denote by $\Omega=\pi (M).$ We have $(0,0)\in \partial \Omega$ and we can write $M=\mbox{graph}(u:\Omega\to \rr)$ with $u(0,0)=0$ and $\nabla u(0,0)=(0,0)$, since they are tangent at $p$.

If $\gamma: I\to \partial \Omega$ is a parametrization of the boundary of $\Omega$, then $\beta(t)=\bigl(\gamma(t), u(\gamma(t))\bigr)$ parametrizes $M\cap \s^2.$ We have 
$$
\beta'(0)=(\gamma'(0), 0),
$$
in particular, $\beta'(0)\in P$ and is tangent to $\s^2$. Thus, up to a change of orientation, $\beta'(0)=\partial_x$ and consequently, $\gamma'(0)=\partial_x,$ that is, $\partial \Omega$ is tangent to $D^+$ at the origin (which is identified with $p$ by our choice of coordinates).


Since $\partial W$ is mean convex and $D^+$ is contained in its mean convex side, we know its mean curvature satisfies $H\geq0$ for the choice of normal $N = \frac{1}{\left(1 + |\nabla u|^2\right)^{1/2}}(u_x,u_y,-1)$. Moreover, $u$ satisfies the mean curvature equation
\begin{equation}\label{MCG}
Lu := \left(1 + u_{y}^2\right)u_{xx} -2u_{x}u_{y}u_{xy} + \left(1 + u_{x}^2\right)u_{yy} = -2\left(1 + |\nabla u|^2\right)^{3/2}H \leq 0.
\end{equation}


Observe that since $\Omega$ is tangent to $D^+$ at the origin, we know $\Omega$ contains a segment of the diameter starting at $p$; hence, choosing a smaller $\Omega$ if necessary, we can assume that $\Omega\cap \{y\mbox{-axis}\}=\{(0, t, 0), -\epsilon < t\leq 0\}$. Let us consider the region $\widetilde{\Omega}=\Omega\cap \{x\geq 0\}$. Since $\pi(\ba^3)$ is exactly the disk containing $D^+,$ we know that $\widetilde\Omega\subset D^+.$ Denote by $\alpha$ the component of the boundary of $\widetilde\Omega$ contained in the $y$-axis. By our choice of coordinates we know that $u\geq 0$ on $\alpha$ and, since we are assuming that $M\cap D^+$ is contained in the diameter of $D^+$, we have $u>0$ on $\widetilde\Omega\cap \{x>0\}$.

Hence, we can apply Serrin's Maximum Principle (see Lemma 2 in \cite{Ser}) to conclude that, since $\nabla u(0,0)=(0,0)$, either $u\equiv 0$ or $\hess u_{(0,0)} (v,v)>0$, for every vector $v$ pointing strictly into $\widetilde\Omega$. (Observe that the operator $L$ defined in \eqref{MCG} satisfies the hypotheses of \cite[Lemma 2]{Ser}).

If $u\equiv0,$ we conclude that $W$ is a halfball and then $M\subset P$, which contradicts our assumption that $M\cap D^+$ is contained in the diameter of $D^+$. Suppose the second option happens, that is, 
 \begin{equation}\label{hess}
\hess u_{(0,0)} (v,v)>0,
 \end{equation}
for any vector $v=(v_1,v_2)$ with coordinates $v_1>0$ and $v_2<0.$ Then, by continuity, we get that
\begin{equation*}
u_{xx}(0,0) = \hess u_{(0,0)}(\p_x,\p_x) \geq 0 \ \ \textrm{and} \ \ u_{yy}(0,0) = \hess u_{(0,0)}(\p_y,\p_y) \geq 0.
\end{equation*}

On the other hand, equation \eqref{MCG} evaluated at $(0,0)$ gives us
\begin{equation*}
u_{xx}(0,0) + u_{yy}(0,0) \leq 0,
\end{equation*}
therefore,
\begin{equation}\label{sec.der}
u_{xx}(0,0) = u_{yy}(0,0) = 0.
\end{equation}

Denote by $\eta$ the outward conormal vector to $\partial M$. Since $M$ meets $\s^2$ orthogonally, $\eta$ coincides with the position vector $\widetilde{\eta}$ along the boundary of $M.$ Since for the $z$-coordinate we know that $\displaystyle\frac{\partial z}{\partial \widetilde\eta}=z;$ restricting to $\partial \Omega,$ we obtain
\begin{equation}\label{Nor.der}
\eta_1u_x+\eta_2u_y=u \ \ \textrm{along} \ \ \partial \Omega,
\end{equation}
where $\eta=(\eta_1,\eta_2)$ in graph coordinates. Since $\partial\Omega$ is tangent to the $x$-axis at $p,$ we can write $\partial\Omega$ in a neighborhood of $p$ as a graph $\bigl(x,h(x)\bigr)$, for some function $h: (-\epsilon,\epsilon) \to \mathbb{R}$. In particular, $h(0)=0$ and $h'(0)=0.$ Hence, we can rewrite \eqref{Nor.der} as
\begin{equation*}
\eta_1\bigl(x,h(x)\bigr)u_x\bigl(x,h(x)\bigr) + \eta_2\bigl(x,h(x)\bigr)u_y\bigl(x,h(x)\bigr) = u\bigl(x,h(x)\bigr), \ \forall \ x\in (-\epsilon,\epsilon).
\end{equation*}

Differentiating the previous equation with respect to $x$ we obtain
\begin{align*}
& \frac{d}{dx}\left[\eta_1\bigl(x,h(x)\bigr)\right]u_x + \eta_1\left[u_{xx} + u_{xy}h'(x)\right] + \frac{d}{dx}\left[\eta_2\bigl(x,h(x)\bigr)\right]u_y + \eta_2\left[u_{xy} + u_{yy}h'(x)\right]\\
&= u_{x} + u_{y}h'(x).
\end{align*}
Thus, when we evaluate at $x = 0$, using \eqref{sec.der} and the facts that $\nabla u(0,0)= (0,0)$ and $\eta(0,0) = (0,1)$, we conclude that $u_{xy}(0,0) = 0$. However, this together with \eqref{sec.der} implies that $\hess u_{(0,0)} \equiv 0$, which contradicts \eqref{hess}.

Therefore, we have $u \equiv 0$ which implies that $M$ is contained in the plane $P$ and, consequently, $W$ is necessarily a halfball.
\end{proof}

\end{document}